\documentclass[12pt]{article}

\usepackage{amssymb,amsmath,amsthm,sectsty,url,mathrsfs,graphicx}
\usepackage[letterpaper,hmargin=1.0in,vmargin=1.0in]{geometry}
\usepackage[dvips,colorlinks,linkcolor=black,citecolor=black,filecolor=black,urlcolor=black]{hyperref}
\usepackage{color}

\sectionfont{\large}
\subsectionfont{\normalsize}
\numberwithin{equation}{section}

\newtheorem{maintheorem}{Theorem}[section]
\newtheorem{theorem}{Theorem}[section]
\newtheorem{lemma}[theorem]{Lemma}
\newtheorem{proposition}[theorem]{Proposition}
\newtheorem{corollary}[theorem]{Corollary}
\newtheorem{fact}[theorem]{Fact}
\newtheorem{definition}[theorem]{Definition}
\newtheorem{remark}[theorem]{Remark}
\newtheorem{conjecture}[theorem]{Conjecture}
\newtheorem{claim}[theorem]{Claim}
\newtheorem{question}[theorem]{Question}

\newtheorem*{open}{Question}

\def\E{\mathop{\mathbb E}}

\newcommand{\PP}{\mathscr{P}}

\renewcommand{\Pr}{ \mathrm P}

\newcommand{ \av}{ \mathrm{ave} }
\newcommand{ \rel}{ t_{\mathrm{rel}} }

\newcommand{ \h}{ \mathrm{H} }
\newcommand{ \TV}{ \mathrm{TV} }

\newcommand{\la}{\lambda}

\newcommand{\ent}{\mathrm{Ent}_{\pi} }
\newcommand{\entt}{\mathrm{Ent} }
\newcommand{\con}{\mathrm{Con} }
\newcommand{\Prr}{\mathbb{P} }
\newcommand{\cT}{\mathcal{T} }

\DeclareMathSymbol{\leqslant}{\mathalpha}{AMSa}{"36} 
\DeclareMathSymbol{\geqslant}{\mathalpha}{AMSa}{"3E} 
\DeclareMathSymbol{\eset}{\mathalpha}{AMSb}{"3F}     
\renewcommand{\leq}{\;\leqslant\;}                   

\DeclareMathOperator*{\Var}{\mathrm Var}      

\renewcommand{\epsilon}{\varepsilon}

\newcommand{\N}{\mathbb N}
\newcommand{\R}{\mathbb R}
\newcommand{\Z}{\mathbb Z}

\usepackage[normalem]{ulem}

\begin{document}

\title{A characterization of  $L_{2}$ mixing and hypercontractivity via hitting times and maximal inequalities}
\author{Jonathan Hermon
\thanks{
Department of Statistics, UC Berkeley, USA. E-mail: {\tt jonathan.hermon@stat.berkeley.edu}.}
\and
Yuval Peres
\thanks{
Microsoft Research, Redmond, Washington, USA. E-mail: {\tt peres@microsoft.com}.}
}
\date{}
\maketitle

\begin{abstract}
There are several works characterizing the total-variation mixing time of a reversible Markov chain in term of natural  probabilistic concepts such as stopping times and hitting times. In contrast, there is no known analog for the $L_{2}$ mixing time, $\tau_{2}$ (while there are sophisticated analytic tools to bound $ \tau_2$, in general they do not determine $\tau_2$ up to a constant factor and they lack a probabilistic interpretation). In this work we show that $\tau_2$ can be characterized up to a constant factor using hitting times distributions. We also derive a new extremal characterization of the Log-Sobolev constant, $c_{\mathrm{LS}}$, as a weighted version of the spectral gap. This characterization yields a probabilistic interpretation of $c_{\mathrm{LS}}$ in terms of a hitting time version of hypercontractivity. As applications of our results, we show that (1) for every reversible Markov chain, $\tau_2$ is robust under addition of self-loops with bounded weights, and (2) for weighted nearest neighbor random walks on trees, $\tau_2 $ is robust under bounded perturbations of the edge weights.
\end{abstract}

\newpage


\section{Introduction}
There are numerous essentially equivalent characterizations of mixing in $L_1$ (e.g.~\cite[Theorem 4.6]{aldous} and \cite{peres2015mixing}) of a finite reversible Markov chain. Some involve natural probabilistic concepts such as couplings, stopping times and hitting times (see \S~\ref{s: related}). In contrast, (paraphrasing Aldous and Fill \cite{aldous} last sentence of page 155, which mentions that there is no $L_2$ counterpart to \cite[Theorem 4.6]{aldous}) while there are several sophisticated analytic and geometric tools for bounding the $L_2$ mixing time, $\tau_2$, none of them has a probabilistic interpretation, and none of them determines $\tau_2$ up to a constant factor.

\medskip

 In this work we provide probabilistic characterizations in terms of hitting times distributions for the $L_2$ mixing time and also for the mixing time  in relative entropy, $\tau_{\entt}$ (see \eqref{eq: taupeps} and \eqref{eq: tauent} for definitions), of a reversible Markov chain (Theorem \ref{thm: mainintro}).

\medskip

While the spectral gap is a natural and simple parameter, the \emph{Log-Sobolev constant} (see \eqref{eq: deftl}), $c_{\mathrm{LS}} $, is a more involved quantity. When one first encounters $c_{\mathrm{LS}} $, it may seem like an artificial parameter that ``magically" gives good bounds on $\tau_2$. We give a new extremal characterization of the Log-Sobolev constant as a weighted version of the spectral gap.   This characterization gives a direct link between $c_{\mathrm{LS}} $ and $\tau_2$ (answering a question asked by James Lee, see Remark \ref{rem:JL}) and can be interpreted probabilistically as a hitting-time version of hypercontractivity (see the discussion following Fact \ref{fact: hyper2}). We note that recently, Cattiaux and Guillin \cite{cat} established a different connection between hitting times distributions and the existence of a Log-Sobolev inequality for diffusions. 

We present our main results in the continuous-time setup.   All of our results can be extended to the setup of discrete-time  chains (see \S~\ref{s:discave}). We note that most of our results can be extended to the general setup of ergodic Markov chains. However, working in such generality leads to many technical difficulties which we chose to avoid for the sake of clarity of presentation.   

\subsection{Characterizations of $\tau_2$ and $\tau_{\entt}$ using hitting times}

We now describe the aforementioned characterizations of $\tau_2$ and $\tau_{\entt}$. More refined versions will be given later on in Theorems \ref{thm: main} and \ref{thm: main3}. Recall that for a Markov chain $(X_t)_{t \ge 0}$ with state space $\Omega$, the \emph{hitting-time} of a set $A \subset \Omega $ is $T_A:=\inf \{t:X_t \in A \} $.  We say that $A$ is connected if $\Pr_{a}[T_b<T_{A^c}]>0$, for all $a,b \in A $.  We denote by $\mathrm{Con}_{\delta}$ the collection of all connected sets $A$ satisfying $\pi(A) \le \delta$, where throughout, $\pi$ shall denote the stationary distribution of the chain. Denote
\begin{equation}
\label{eq: rhorhoent}
\rho:=\max_{x \in \Omega}\rho_x \quad \text{and} \quad \rho_{\entt}:= \max_{x \in \Omega}\rho_{\entt,x}, \quad \text{where}
\end{equation}
\begin{equation}
\label{eq: rhosubx}
\begin{split}
& \rho_{x}:= \min \{t: \Pr_{x}[T_{A^c}>t] \le \pi(A)+\frac{1}{2} \sqrt{ \pi(A)\pi(A^c)} \text{ for all } A \in \mathrm{Con}_{1/2}  \},
\\ & \rho_{\entt,x}:= \min \{t: \Pr_{x}[T_{A^c}>t] \le \min \left( \frac{C_{\entt}}{|\log \pi(A)|},\frac{99}{100}  \right),   \text{ for all } A \in \mathrm{Con}_{1/2}  \},
 \end{split}
\end{equation}
for some absolute constant $C_{\entt} >0 $ to be determined later (in \S~\ref{s:ideaslower} in the discussion following \eqref{eq: Lagrange}). Note that allowing $A  $ above to range over all $A \subset \Omega$ such that $\pi(A) \le 1/2$ does not change the values of $\rho_x $ and $\rho_{\entt,x}$.
\begin{maintheorem}
\label{thm: mainintro}
There exist absolute constants $C_{1},C_2,C_3$ such that for every irreducible reversible Markov chain on a finite state space
\begin{equation}
\label{eq: main3intro}
 \quad \rho \le \tau_{2} \le \rho + C_1/c_{\mathrm{LS}}  \le C_{2} \rho.
\end{equation}
\begin{equation}
\label{eq: main3'intro}
 \rho_{\entt} \le \tau_{\entt}  \le C_{3} \rho_{\entt}.
\end{equation}
\end{maintheorem}

Note that in the definitions of $\rho$ and $\rho_{\entt}$, the smaller $A$ is, the smaller we require the chance of not escaping it by time  $\rho$ or $\rho_{\entt}$, respectively, to be. In other words, the smaller $A$ is,
the higher the ``penalty" we assign to the case the chain did not escape from it. As we explain in \S~\ref{s:ideaslower}, the first inequalities in \eqref{eq: main3intro}-\eqref{eq: main3'intro} are easy and even somewhat  ``naive".

\subsection{A new extremal characterization of the Log-Sobolev time.}

A lot of attention has been focused on inequalities that interpolate between the
Log-Sobolev inequality and the Poincar\'e (spectral gap) inequality (e.g.~\cite{mod,reverse}). Using similar ideas as described above we prove a new extremal characterization (up to a constant factor) of the \emph{Log-Sobolev constant} (Theorems \ref{thm: kappa=tell}), $c_{\mathrm{LS}} $ (see \eqref{eq: deftl} for a definition). The \emph{Log-Sobolev time} is defined as $t_{\mathrm{LS}}:=1/c_{\mathrm{LS}} $.

 The aforementioned characterization has a relatively simple form which does not involve any entropy. Instead, it describes the Log-Sobolev constant as a weighted version of the spectral gap. This characterization provides some insights regarding the hierarchy of the aforementioned inequalities. Before presenting it, we first need a few definitions.

\medskip

The \emph{time-reversal} of $P$ is defined as $P^*(x,y):=\pi(y)P(y,x)/\pi(x)$. This is the dual operator of $P$ w.r.t.~$L_2(\Omega, \pi )$. We say that $P$ is \emph{reversible} if $P=P^*$. Denote $Q:=(P+P^*)/2$. Note that $Q=Q^*$. The \emph{spectral gap} of $P$, denoted by $\la$, is defined as the smallest non-zero eigenvalue of $I-Q$. The relaxation-time is defined as $\rel:=1/\la$. Let $A \subsetneq\ \Omega $.  Let  $Q_A$ (resp.~$P_A$) be the restriction of $Q$ (resp.~$P$) to $A$. Note that $Q_A$ and $P_A$ are substochastic. The spectral gap of $P_A$, denoted by $\la(A)$, is defined as the minimal eigenvalue of $I-Q_A$. Denote $\rel(A):=1/\la(A)$. Denote
\begin{equation}
\kappa:=1/\alpha, \quad \alpha:=\min_{A \in \con_{1/2} }\alpha(A), \quad \text{where} \quad  \alpha(A):=\la(A)/|\log \pi(A) |.
\end{equation}
As mentioned earlier, $\alpha$ is a weighted version of $\la$ since (\cite[Lemma 4.39]{aldous} and \cite[(1.4)]{spectral})
\begin{equation}
\label{eq: laAla}
\la /2 \le  \min_{A \in \con_{1/2} }\la(A) \le \la, \quad \text{and so} \quad \rel \log 2 \le \kappa.
\end{equation}
\begin{maintheorem}
\label{thm: kappa=tell}
For every irreducible Markov chain on a finite state space
\begin{equation}
\label{eq: kappa=tell}
\kappa \le t_{\mathrm{LS}} \le 2(\kappa +  \rel (1+ \log 49)) \le 2(1+(1+ \log 49) /\log 2 ) \kappa < 17 \kappa .
\end{equation}
\end{maintheorem}
\begin{remark}
The inequality $\kappa \le t_{\mathrm{LS}}  $ is easy. See Lemma 4.2 in \cite{spectral} for a stronger inequality. The harder and more interesting direction is $t_{\mathrm{LS}} \le C \kappa$, which is an  improvement over the well-known inequality $t_{\mathrm{LS}} \le \rel \frac{\log [1/\pi_* - 1]}{1-2\pi_*} $, where $\pi_*:=\min_{x \in \Omega} \pi(x) $ \cite[Corollary A.4]{diaconis}.
\end{remark}
\begin{remark}
\label{rem:JL}
Despite the fact that $t_{\mathrm{LS}}$ is a geometric quantity, Logarithmic Sobolev inequalities have a strong analytic flavor and little probabilistic interpretation. For instance, the proof of the inequality $t_{\mathrm{LS}}\le 2\tau_2(1/e)$ \cite[Corollary 3.11]{diaconis} (where $\tau_2(\epsilon) $ is the   $L_2$ mixing time defined in \eqref{eq: taupeps}) relies on Stein's interpolation Theorem for a family of analytic operators. Our analysis yields a probabilistic proof of the fact that $t_{\mathrm{LS}}\le C \tau_{2}$ for reversible chains.
The problem of finding such a proof was posed by James Lee at  the Simons institute in 2015.
Indeed by Theorem \ref{thm: kappa=tell} and  \eqref{eq: com1},  $t_{\mathrm{LS}}/17 \le  \kappa \le 3 \rho \le 3 \tau_2  $. The second inequality is relatively easy, and is obtained by analyzing hitting times, rather than by analytic tools. As we show in \S~\ref{s:ideaslower},  the inequality $\rho \le \tau_2$   also has a probabilistic interpretation.
\end{remark}
We note that while some effort was made to make most constants explicit in order to demonstrate that they are not large, we did not attempt to optimize constants. We use the convention that $C,C',C_1,\ldots$ (resp.~$c,c',c_1,\ldots $) denote positive absolute constants which are sufficiently large (resp.~small). Different appearances of the same constant at different places may refer to different numeric values.
\subsection{Organization of this work}
  In \S~\ref{s:app} we present some applications of our main results, two of which concern robustness of mixing times and one (\S~\ref{s:discave}) concerns a comparison result between the ($L_2$ and relative entropy) mixing times of the discrete-time averaged chain (defined in \S~\ref{s:discave}) and of the continuous-time chain. In \S~\ref{s:pre}  we provide some background about mixing-times, the spectral gap and the Log-Sobolev constant and present some auxiliary results about maximal inequalities and hitting times. In \S~\ref{s:ideaslower} we prove the lower bounds on $\tau_2$ and $\tau_{\entt}$ from \eqref{eq: main3intro} and \eqref{eq: main3'intro} and in \S~\ref{s:ideasupper} we present a sketch of the proof of the upper bound on $\tau_2$ from \eqref{eq: main3intro}.  In \S~5 we prove our main results (we prove Theorem \ref{thm: kappa=tell} and also slightly more refined versions of the upper bounds from \eqref{eq: main3intro} and \eqref{eq: main3'intro}, resp.). The necessary adaptations for the discrete-time setup are given at \S~\ref{s:disc}.  In \S~6 we prove the two applications from \S~\ref{s:app} concerning robustness of mixing times (Corollary \ref{cor:lazinessrobustness} and Theorem \ref{trees}). We conclude with some open problems in \S~7.
\section{Applications}
\label{s:app}
\subsection{Robustness of $\tau_2$ under addition of self-loops of bounded weights.}

\begin{corollary}
\label{cor:lazinessrobustness}
Let $(X_t)$ be a reversible irreducible continuous-time Markov chain on a finite state space $\Omega$  with generator $G$. Let  $(\tilde X_t)$ be a chain with generator $\tilde G $ obtained by multiplying for all $x \in \Omega $ the $x$th row of $G$ by some $r_x \in (1/M,M)$ (for some $M \ge 1$). Then for some absolute constant $C $ the corresponding $L_2$ and relative-entropy mixing times satisfy
\begin{equation}
\label{eq: laziness}
\tilde \tau_2 /(CM \log M ) \le \tau_2 \le (CM \log M) \tilde \tau_2.
\end{equation}
\begin{equation}
\label{eq: laziness2}
\tilde \tau_{\entt} /(CM \log M ) \le \tau_{\entt} \le (CM \log M) \tilde \tau_{\entt}.
\end{equation}
\end{corollary}
This corollary, proved in \S~\ref{s:rob}, is an analog of \cite[Corollary 9.5]{peres2015mixing}, which gives the corresponding statement for $\tau_1$. While the statement  is extremely intuitive, surprisingly, it was recently shown that it may fail for simple random walk on an Eulerian digraph \cite[Theorem 1.5]{Perla}.

Observe that the generator $G$ of a reversible chain on a finite state space $\Omega$, can be written as $r(P-I)$, where $P$ is the transition matrix of some nearest neighbor weighted random walk on a network which may contain some weighted self-loops. The operation of multiplying the $x$th row of $G$ by some $r_x \in (1/M,M)$ for all $x \in \Omega$ is the  same as changing $r$ above by some constant factor and changing the weights of the self-loops by a constant factor.
\begin{remark}
Similarly, one can show that under reversibility the $L_2$ mixing time in the discrete-time lazy setup is robust under changes of the holding probabilities. More precisely, for every $\delta \in (0, 1/2] $ if we consider a chain that for all $x \in \Omega$, when at state $x$ it stays put w.p.~$\delta \le a(x) \le 1-\delta $ and otherwise moves to state $y$ w.p.~$P(x,y)$ (where $P$ is reversible), then its $L_2$ mixing time can  differ from the $L_2 $ mixing time of the chain with $a(x)=1/2$ for all $x$ only by a factor of $C \delta^{-1}|\log \delta |$.
\end{remark}
\subsection{Robustness of $\tau_{\infty}$ for trees.}

Recall that for reversible chains the $L_2$ mixing time, $\tau_2$, determines the $L_p$-mixing time up to a factor $c_p$ for all $1 < p \le \infty $   (see \eqref{eq: L2Linfty}). Denote the $L_p$ mixing time of simple random walk on a finite connected simple graph $G$ by $\tau_p(G)$. Kozma \cite{kozma} made the following conjecture:
\begin{conjecture}[\cite{kozma}]
\label{con: Kozma}
Let $G$ and $H$ be two finite $K$-roughly isometric graphs of maximal degree $\le d$. Then
\begin{equation}
\label{eq: Gady}
\tau_{\infty}(G) \le C(K,d) \tau_{\infty}(H).
\end{equation}
\end{conjecture}
It is well-known that \eqref{eq: Gady} is true if one replaces $\tau_{\infty} $ with $t_{\mathrm{LS}}$ (e.g.~\cite[Lemma 3.4]{diaconis}). Ding and Peres \cite{ding} showed that \eqref{eq: Gady} is false if one replaces $\tau_{\infty} $ with $\tau_1$. In part, their analysis relied on the fact that the total variation mixing time can be related to hitting times, which may be sensitive to small changes in the geometry. Hence it is natural to expect that a description of $\tau_{\infty}$ in terms of hitting times might shed some light on Conjecture \ref{con: Kozma}. Indeed this was one of the main motivations for this work. In \cite{hermon} the first author constructed a counterexample to Conjecture \ref{con: Kozma}, where also there the key is sensitivity of hitting times.

\medskip

Peres and Sousi \cite[Theorem 9.1]{peres2015mixing} showed that for weighted nearest neighbor random walks on trees (see \S~\ref{s: trees} for a definition), $\tau_1 $ can change only by a constant factor, as a result of a bounded perturbation of the edge weights.
As an application of Theorem \ref{thm: mainintro} we extend their result to the case of $\tau_2$.
\begin{maintheorem}
\label{trees}
There exists an absolute constant $C$ such that for every finite tree $\mathcal{T}=(V,E)$  with some edge weights $(w_e)_{e \in E}$, the corresponding random walk satisfies that
\begin{equation}
\label{eq: trees}
 \max (\tau_1,t_{\mathrm{LS}}/4) \le \tau_{2} \le \tau_1+C \max (t_{\mathrm{LS}},\sqrt{t_{\mathrm{LS}}\tau_1} ) ,
\end{equation}
Consequently, if $(w'_e)_{e \in E},(w_e)_{e \in E} $ are two edge weights such that $1/M \le w_e/w'_e \le M $ for all $e \in E$, then there exists a constant $C_M$ (depending only on $M$) such that the corresponding $L_{\infty}$ mixing times, $\tau_{\infty} $ and $\tau'_{\infty} $,  satisfy
\begin{equation}
\label{eq: trees2}
\tau_{\infty}'  /C_M \le \tau_{\infty} \le C_{M}  \tau_{\infty}'  .
\end{equation}
\end{maintheorem}
\begin{remark}
Since $t_{\mathrm{LS}}$ is robust under a bounded perturbation of the edge weights (e.g.~\cite[Lemma 3.3]{diaconis}), indeed \eqref{eq: trees2} follows from \eqref{eq: trees} in conjunction with the aforementioned $L_1$ robustness of trees (and the fact that $\tau_2 \le \tau_{\infty} \le 2 \tau_2$, see \eqref{eq: L2Linfty}).
\end{remark}
\subsection{Comparison of continuous-time, discrete-time and averaged chains}
\label{s:discave}
Let $(X_{k})$ be a finite irreducible reversible discrete-time Markov chain with transition matrix $P$. Since reversible Markov chains can only have period 2, one may wonder whether it suffices
to average over two consecutive times (i.e.\ to make a single lazy step) in order to avoid near-periodicity issues.  
This motivates considering the following Markov chain. For any $t \ge 1$, denote $A_t:=(P^t+P^{t-1})/2$. The \emph{\textbf{averaged chain}},
 $(X_t^{\mathrm{ave}})_{t
= 0}^{\infty}$,  with initial
state $x$, is a  Markov chain, whose distribution at time $t
\ge 1$ is $A_t(x,\cdot)$, where $A_{t}(x,y):=
(P^{t}(x,y)+P^{t-1}(x,y))/2$. Equivalently, $(X_t^{\mathrm{ave}})_{t
= 1}^{\infty}:=(X_{t-\xi})_{t
= 1}^{\infty}$, where $\xi$ is a  $\mathrm{Bernoulli}(1/2)$ random variable,
 independent of $(X_t)_{t
= 0}^{\infty}$.

\medskip

 We may consider the $L_p$-mixing times of the discrete-time and averaged chains $\tau_{p}^{\mathrm{discete}}(\cdot) $ and $\tau_p^{\av}(\cdot)$, resp., defined in an analogous manner as $\tau_p(\cdot) $, obtained by replacing $h_t(x,y)=H_t(x,y)/\pi(y) $ with $k_t(x,y):=P^t(x,y)/\pi(y) $ and $a_t(x,y):=A_t(x,y)/\pi(y)$, resp. (see \S~\ref{s:def}). Similarly, we may consider the relative-entropy mixing times of the discrete-time and averaged chains $\tau_{\entt}^{\mathrm{discrete}}(\cdot) $ and $\tau_{\entt}^{\av}(\cdot)$, resp.. We define $\rho_{{\mathrm{discete}}} $ and $\rho_{\entt}^{{\mathrm{discete}}} $ in an analogous manner to $\rho$ and  $\rho_{\entt}$, where now the hitting times are defined w.r.t.~the discrete-time chain. Denote the eigenvalues of $P$ by $1= \la_1 > \la_2 \ge \cdots \ge \la_{|\Omega|} \ge -1 $.  Define $\rel^{\mathrm{absolute}}:=\max \{| \log |\lambda_2||^{-1},|\log|\lambda_{|\Omega |}||^{-1} \}$. Let $\beta(A):=1-\la(A)$ and  \[\kappa_{\mathrm{discrete}} := \max_{A \in \con_{1/2}}\log_{\frac{1}{\beta(A)}}\left(\frac{1}{\pi(A)}\right) \le \kappa . \]
\begin{theorem}
\label{thm: aveL2}
There exist positive absolute constants $c,C_{1},C_2,C_3$ such that for every irreducible reversible Markov chain on a finite state space
\begin{equation}
\label{eq: aveL2}
 \quad \rho_{{\mathrm{discete}}} \le \tau_{2}^{\av} \le \rho_{{\mathrm{discete}}} + C_1\kappa_{\mathrm{discrete}}\le C_{2} \rho_{{\mathrm{discete}}}.
\end{equation}
\begin{equation}
\label{eq: avent}
 \rho_{\entt}^{{\mathrm{discete}}} \le \tau_{\entt}^{\av}  \le C_{3} \rho_{\entt}^{{\mathrm{discete}}}.
\end{equation}
\begin{equation}
\label{eq: discL2}
 \max ( \rho_{{\mathrm{discete}}},\rel^{\mathrm{absolute}}\log 2) \le \tau_{2}^{\mathrm{discete}}\le \rho_{{\mathrm{discete}}} + C_1(\kappa_{\mathrm{discrete}}+\rel^{\mathrm{absolute}})  \le C_{2} (\rho_{{\mathrm{discete}}}+\rel^{\mathrm{absolute}}).
\end{equation}
\begin{equation}
\label{eq: discent}
\max( \rho_{\entt}^{{\mathrm{discete}}},c\rel^{\mathrm{absolute}})\le \tau_{\entt}^{\mathrm{discrete}}  \le C_{3} (\rho_{\entt}^{{\mathrm{discete}}}+\rel^{\mathrm{absolute}}).
\end{equation}
Moreover, there exists an absolute constant $M \ge 1/2 $ such that
\begin{equation}
\label{eq:disc4}
\tau_{2}^{\mathrm{discete}}(M) \le \rho_{{\mathrm{discete}}} + C_1\kappa_{\mathrm{discrete}}.
\end{equation}

\end{theorem}
In conjunction with Theorem \ref{thm: mainintro} and Lemma \ref{lem: rhodiscrhocts}, which asserts that $ \rho_{{\mathrm{discete}}} \le \bar C \rho  $ and $ \rho_{\entt}^{{\mathrm{discete}}} \le \bar C' \rho_{\entt} $, we get the following corollary.
\begin{corollary}
\label{cor:L2equiv}
There exists an absolute constant $C>0$ such that for every irreducible reversible Markov chain on a finite state space
\[ \tau_{2}^{\av}  \le C \tau_2 \quad \text{and} \quad  \tau_{\entt}^{\av} \le C\tau_{\entt}. \]
\end{corollary}
\begin{remark}
To see that the reverse inequalities are false consider simple random walk (SRW) on the $n$-clique, for which $\tau_{2}^{\av} \le 2 $ while $\tau_2 = \Theta(\log n) $ and  $\tau_{\entt} = \Theta(\log \log n) $. Indeed, it is possible that $\rho^{\mathrm{discrete}} \ll \rho$. Loosely speaking, this type of behavior is possible when $\max (|\lambda_{|\Omega|}|, |\la_2|) \ll 1 $ (e.g.~for SRW on the $n$-clique $\la_2=\la_n=-\frac{1}{n-1} $). More generally, consider an arbitrary distribution $\pi$ and the transition matrix $\Pi $ whose rows are all equal to $\pi$. Then $\tau_{\infty}^{\mathrm{discrete}}=1$ while $\tau_{\infty} \approx | \log (2\min_{x}\pi(x))| $. Note that $\la_i=1_{i=1}$. 
\end{remark}
\begin{remark}
The fact that under reversibility $\tau_{1}^{\av} \le C \tau_1 $ is due to Peres and Sousi  \cite{peres2015mixing}. In fact,  in \cite{AF} the authors confirmed a conjecture by  Aldous and Fill \cite[Open Problem 4.17]{aldous} by showing that under reversibility, for all $t,M \ge e$ and $x \in \Omega$
\begin{equation*}
\label{eq:AFeasy}
\|H_{t+M \sqrt{t}}(x,\cdot) -\pi(\cdot)\|_{\mathrm{TV}}-e^{-cM^2} \le  \|A_{t }(x,\cdot) -\pi (\cdot) \|_{\mathrm{TV}} \le  \|H_{t-( M \log M) \sqrt{t}}(x,\cdot) -\pi(\cdot)\|_{\mathrm{TV}}  + C/M,
\end{equation*}
where $\|\mu - \nu \|_{\TV}=\frac{1}{2}\|\mu - \nu \|_{1,\pi}=\frac{1}{2}\sum_x |\mu(x)-\nu(x)|$ and $H_s:=e^{-s(I-P)}$.  
\end{remark}
The following proposition refines the inequality $\tau_2^{\mathrm{ave}} \le C \tau_2 $. The argument is borrowed from \cite[Corollary 2.2]{diaconis}.

\begin{proposition}
\label{prop:avcts}
For every finite irreducible reversible Markov chain, for every $k \ge 2$
\[\forall k' \ge 1, \quad \|A_{k+k'}(x,\cdot)- \pi(\cdot) \|_{2,\pi}^2 \le \left(\frac{1}{2ek'} \right)^2 (\|\Pr_x^{k-2} - \pi \|_{2,\pi}^2+1)+(1- \la)^{2k'+2} \|\Pr_x^{k-2}-\pi \|_{2,\pi}^2. \]
\end{proposition}
\section{Propaedeutics}
\label{s:pre}
\subsection{Different notions of distance and mixing times and their relations}
\label{s:def}
Generically, we shall denote the state space of a Markov chain  $(X_t)$  by $\Omega
$ and its stationary distribution by $\pi$. We denote such a chain by $(\Omega,P,\pi)$. We say that the chain is finite, whenever $\Omega$ is finite.
The continuous-time version of a chain
is a continuous-time Markov chain whose distribution at time $t$ is given by the heat kernel $H_t:=e^{-t(I-P)}$. We denote $h_{t}(x,y):=H_{t}(x,y)/\pi(y)$.   

\medskip

We denote by $\Pr_{x}^t$ (resp.~$\Pr_{x}$) the distribution of $X_t$ (resp.~$(X_t)_{t \ge 0 }$), given that the initial state is $x$. The $L_p$ norm of a function $f \in \R^{\Omega}$ is $\|f\|_p:=(\mathbb{E}_{\pi}[|f|^{p}])^{1/p}$ for $1 \le p < \infty$ (where $\mathbb{E}_{\pi}[h]:= \sum_x \pi(x) h(x)$) and $\|f\|_{\infty}:=\max_x |f(x)|$.
The $L_p$ norm of a signed measure $\sigma$ is
\begin{equation*}
\label{eq: Lpdef}
\|\sigma \|_{p,\pi}:=\|\sigma / \pi \|_p, \quad \text{where} \quad (\sigma / \pi)(x)=\sigma(x) / \pi(x).
\end{equation*}
We denote the worst case $L_p$ distance at time $t$   by $d_{p}(t):=\max_x d_{p,x}(t)$, where $d_{p,x}(t):= \|\Pr_x^t-\pi \|_{p,\pi}$. Under reversibility for all $x \in \Omega$ and $t \ge 0$ (e.g.~(2.2) in \cite{spectral})
\begin{equation}
\label{eq: generalLp}
d_{2,x}^2(t)= h_{2t}(x,x)-1, \quad d_{\infty}(t)=\max_{y} h_{t}(y,y)-1.
\end{equation}
The $\epsilon$-$L_{p}$\textbf{-mixing-time} of the chain (resp.~for a fixed starting state $x$) is defined as
\begin{equation}
\label{eq: taupeps}
\tau_{p}(\epsilon):= \max_{x} \tau_{p,x}(\epsilon), \quad \text{ where } \tau_{p,x}(\epsilon):= \min \{t:  d_{p,x}(t)\le \epsilon \}.
\end{equation}
When $\epsilon=1/2$ we omit it from the above
notation. Let  $m_p:=1+ \lceil (2-p)/(2(p-1)) \rceil$.
 It follows from \eqref{eq: generalLp},  Jensen's inequality and the Reisz-Thorin interpolation Theorem that for reversible chains, the $L_p$ mixing times can be compared as follows (e.g.~\cite[Lemma 2.4.6]{saloff}):
\begin{equation}
\label{eq: L2Linfty}
\begin{split}
  \tau_2(a) \le & \tau_{p}(a) \le 2 \tau_{2}(\sqrt{a})=\tau_{\infty}( a) \quad \text{ for all } p\in (2,\infty] \text{ and }a>0,
\\  \frac{1}{m_p} \tau_2(&a^{m_p}) \le  \tau_p(a)\le   \tau_2(a)  \quad \quad \, \, \text{ for all } p\in (1,2) \text{ and }a>0,
\end{split}
\end{equation}
Hence for all $1<p \le \infty$ the $L_{p}$ convergence profile is determined by that of $L_2$.

\medskip

\medskip

The \textbf{\emph{relative entropy}} of a distribution $\mu$ w.r.t.~$\pi$ is defined as
\begin{equation}
\label{eq: REdist}
D( \mu || \pi ):=\sum_x \mu(x) \log (\mu(x)/\pi(x))=\mathrm{Ent}_{\pi}(\mu / \pi), \quad \text{where for }f \in \R_{+}^\Omega
\end{equation}
\[\mathrm{Ent}_{\pi}(f):=\mathbb{E}_{\pi}[f\log f]-\mathbb{E}_{\pi}[f] \log \mathbb{E}_{\pi}[f]=\mathbb{E}_{\pi}[f\log (f/\mathbb{E}_{\pi}[f])]. \]
The mixing time in relative entropy is defined as
\begin{equation}
\label{eq: tauent}
\tau_{\entt,x}:=\inf \{t: D( \Pr_x^t ||\pi) \le 1/2 \} \quad \text{and} \quad \tau_{\entt}=\max_x \tau_{\entt,x}.
\end{equation}
The relative entropy distance can be compared with the $L_1$ and $L_2$ distances as follows: \cite[p.~110-112]{mathai}
 \begin{equation}
 \label{eq:entL1}
 2D( \mu || \pi ) \ge  \|\mu -\pi\|_{1,\pi}^2+(M^{-1}\|\mu -\pi\|_{1,\pi})^M \quad \text{for some constant} \quad M \ge 3,
 \end{equation}
   and (\cite[Theorem 5]{su})
\begin{equation}
\label{eq: su}
D( \mu || \pi ) \le \log (1+\|\mu-\pi\|_{2,\pi}^2).
\end{equation}
\subsection{Background on the spectral-gap and the Log-Sobolev constant}

The following fact (often referred to as the Poincar\'e inequality) is standard. It can be proved by elementary linear-algebra using the spectral decomposition (e.g.~\cite[Lemma 3.26]{aldous}).
\begin{fact}
\label{lem: contraction}
Let $(\Omega,P,\pi)$ be a finite   irreducible  Markov chain.
Let
$x \in \Omega  $  and $s,t \ge 0$. Then
\begin{equation}
\label{eq: L2contraction}
\|\Pr_x^{t+s}-\pi \|_{2,\pi} \le   e^{-s/\rel}
\|\Pr_x^{t}-\pi \|_{2,\pi}.
\end{equation}
 In particular, for all $x \in \Omega$ and $M \ge 1$,  $$\tau_{2,x} \le \tau_{2,x}(M/2)+\rel \log M. $$
\end{fact}
The relaxation-time provides a lower bound on $L_1$ mixing time as follows. Let $\beta \neq 1$ be an eigenvalue of $P$. Then under reversibility (\cite[Theorem 12.4 and Lemma 20.11]{levin})  
\[\|P^k(x,\cdot)-\pi(\cdot) \|_{1,\pi} \ge  |\beta|^{k} \quad \text{and} \quad \|\Pr_x^{t}-\pi \|_{1,\pi} \ge e^{-(1-\beta)t}. \]
Consequently, under reversibility, for all $\delta \in (0,1] $ we have that
\begin{equation}
\label{eq: lowerL1rel}
\tau_1^{\mathrm{discrete}}(\delta) \ge \rel^{\mathrm{absolute}} \log (1/\delta) \quad \text{and} \quad \tau_1(\delta) \ge \rel \log (1/\delta). 
\end{equation}
It follows from \eqref{eq:entL1} that there exists some absolute constant $c>0$ such that
\begin{equation}
\label{eq: lowerRErel}
\tau_{\entt}^{\mathrm{discrete}} \ge c \rel^{\mathrm{absolute}}. 
\end{equation}
With the convention $0 \log 0=0$, for all non-zero $f,g \in \R_{+}^{\Omega}$ we
\label{s: spectral}
define $\langle f,g \rangle_{\pi}:=\mathbb{E}_{\pi}[fg]$,
  $\mathcal{E}(f,g):=\langle (I-Q) f,g \rangle_{\pi} $ and $\mathcal{E}(f):=\mathcal{E}(f,f) $. The \emph{Log-Sobolev constant}  of the chain is
\begin{equation}
\label{eq: deftl}
c_{\mathrm{LS}}:= \inf \{\mathcal{E}(f)/\ent(f^{2}) : f \text{ is non-constant}  \}.
 \end{equation}
Recall that $t_{\mathrm{LS}}:=1/c_{\mathrm{LS}} $.
It is always the case that $t_{\mathrm{LS}}\ge 2\rel $ (e.g.~\cite[Lemma 3.1]{diaconis}).

\medskip

There are numerous works aiming towards general geometric upper  bounds on $\tau_{\infty}$. Among the most advanced techniques are the  spectral profile \cite{spectral} and Logarithmic Sobolev inequalities (see \cite{diaconis} for a survey on the topic). Let  $\pi_*:=\min_{x \in \Omega}\pi(x)$.  It is classical (e.g.~\cite[Corollary 3.11]{diaconis}) that for reversible chains \begin{equation}
\label{eq: classic}
t_{\mathrm{LS}}/2 \le \tau_2(1/e) \le t_{\mathrm{LS}}(1+\frac{1}{4}\log \log (1/\pi_*)).  \end{equation}
There are examples demonstrating that each of these bounds can be attained up to a constant factor.

Let $1 \le p_1,p_2 \le \infty$. The $p_{1} \to p_2$ norms of a linear operator $\mathbf{A}$ are given by
$$\|\mathbf{A}\|_{p_{1} \to p_{2}}:=\max \{\|\mathbf{A}f\|_{p_{2}}:\|f\|_{p_1}=1 \}.  $$
If $\|\mathbf{A}\|_{p_{1} \to p_{2}} \le 1 $ for some $1 \le p_1 < p_2 \le \infty$ we say that $\mathbf{A} $ is a hypercontraction. For all  $ p_1,p_2 $, $\|H_t\|_{p_1 \to p_2}$ is non-increasing in $t$. It is a classic result (e.g.~\cite[Theorem 3.5]{diaconis} and \cite[Theorem 8.24]{aldous}) that the Log-Sobolev time can be characterized in terms of hypercontrativity.
\begin{fact}
\label{fact: hyper}
Let $(\Omega,P,\pi)$ be a finite reversible chain. Let $s_q:=\inf \{t:\|H_t\|_{2 \to q} \le 1 \}$. Then $t_{\mathrm{LS}}=4\sup_{q:2<q<\infty}s_q/\log (q-1)$.
\end{fact}
The following result (\cite[Theorem 3.10]{diaconis}) will allow us to bound $t_{\mathrm{LS}}$ from above.
\begin{fact}
\label{fact: hyper2}
Let $(\Omega,P,\pi)$ be a finite reversible  chain.  Fix $2<q<\infty$. Assume that $r_q$ and $M_q$ satisfy that $\|H_{r_{q}} \|_{2 \to q } \le M_q $. Then
\begin{equation}
\label{eq: hyper}
t_{\mathrm{LS}} \le \frac{2q}{q-2} r_{q}+2 \rel (1+\frac{q}{q-2} \log M_q ).
\end{equation}
\end{fact}

Fix some  $0<\epsilon < 1/2$ and $A \in \con_{2^{-1/\epsilon}}$.  Assume that $\Pr_{\pi}[T_{A^c}>t] \ge 2 \pi(A)^{1+\epsilon} $.   Let $\pi_A$ denote $\pi$ conditioned on $A$ (i.e.~$\pi_A(a)=\frac{\pi(a)1_{a \in A}}{\pi(A)} $). Then  $\Pr_{\pi_{A}}[T_{A^c}>t] \ge 2 \pi(A)^{\epsilon} $  and so \[B=\{a \in A: \Pr_{a}[T_{A^c}>t] \ge \pi(A)^{\epsilon} \} \] satisfies $\pi_{A}(B) \ge  \pi(A)^{\epsilon} $ (i.e.~$\pi(B)\ge \pi(A)^{1+\epsilon} $). Consequently, for $q>\frac{2(1+\epsilon)}{1-2\epsilon} $
\[ \|H_t 1_A \|_{q} \ge [\sum_{b \in B}\pi(b)H_{t}(b,A)^q]^{1/q} \ge \pi(B)^{1/q}\pi(A)^{\epsilon } \ge \pi(A)^{\epsilon+(1+\epsilon)/q} > \sqrt{ \pi(A)}=\|1_A\|_2. \]
Thus a natural hitting time version of hypercontractivity is
\[ t_{\mathrm{ht}}:=\min \{t: \Pr_{\pi}[T_{A^c}>t] \le \pi(A)^{5/4} \text{ for all }A \in \con_{1/2} \}.  \]
\begin{open}
\label{con: escape}
Is there an absolute constant $C$ such that for every finite irreducible reversible Markov chain $  t_{\mathrm{ht}}/C \le t_{\mathrm{LS}} \le C t_{\mathrm{ht}} $.
\end{open}
Trivially, $t_{\mathrm{ht}}=\min \{t: \Pr_{\pi_{A}}[T_{A^c}>t] \le \pi(A)^{1/4} \text{ for all }A \in \con_{1/2} \} $. Note that if we replace $\pi_A $  by the quasi-stationary distribution of $A$, denoted by $\mu_A$, then by \eqref{eq: stationary<quasi} we get precisely $\kappa/4$. This explains why also $\kappa$ can be interpreted as a hitting time version of hypercontractivity. We note that the above question resembles Open problem 4.38 in \cite{aldous}, which asks whether for reversible chains $\rel \le C \max_{A \in \con_{1/2}} \mathbb{E}_{\pi_A}[T_{A^c}]$, where indeed \cite[Lemma 4.39]{aldous}  $\rel \le  \max_{A \in \con_{1/2}} \mathbb{E}_{\mu_A}[T_{A^c}]$ (the formulation in \cite{aldous} is slightly different, but it is equivalent to our formulation).

\subsection{Starr maximal inequality and a useful lemma}
\label{s:Starr}
In this section we prove a maximal inequality which shall be central in what comes. Denote $S_t:=e^{-(I-Q)t}=\sum_{k=0}^{\infty}\frac{e^{-t}t^k}{k!}Q^t$. When considering $Q$ instead of $P$ we write $\mathbb{P}_x^t$,  $\mathbb{P}_x$ and $Y_t$ instead of  $\Pr_{x}^t$,  $\Pr_{x}$ and $X_t$, respectively.

\begin{theorem}[Starr's Maximal inequality \cite{starr}]
\label{thm: maxergodic}
Let $(\Omega,P,\pi)$ be an irreducible Markov chain. Let $f \in \R^{\Omega}
$. Its corresponding {\em maximal function}  $f^* \in \R^\Omega$  is defined as
$$f^{*}(x):=\sup_{0 \le t < \infty}|S_{t}(f)(x)|=\sup_{0 \le t < \infty}|\mathbb{E}_{x}[f(Y_{t})]|. $$
Then for every  $1<p<\infty$
\begin{equation}
\label{eq: ergodic1}
 \|f^{*} \|_{p} \le p^{*}\|f \|_p, \quad \text{where} \quad p^{*}:=p/(p-1) \text{ is the conjugate exponent of }p.
\end{equation}
Moreover, under reversibility, for $f_{*,\mathrm{even}}(x):=\sup_{ k \in \Z_{+} }|P^{2k}f(x)| $ we have that \eqref{eq: ergodic1} holds also with $f_{*,\mathrm{even}} $ in the role of $f^*$ and hence $f_*(x):=\sup_{ k \in \Z_{+} }|P^{k}f(x)|$ satisfies for $1<p<\infty$ \begin{equation}
\label{eq: Starrdisc}
\|f_{*} \|_{p}^p \le \|f_{*,\mathrm{even}} \|_{p}^p  +  \|(Pf)_{*,\mathrm{even}} \|_{p}^p  \le (p^{*})^{p}(\|f \|_p^p+  \|Pf \|_{p}^p) \le 2(p^{*})^{p}\|f \|_p^p  . \end{equation}
\end{theorem}
The following  Lemma is essentially due to Norris, Peres and Zhai \cite{norris2015surprise}.
\begin{lemma}
\label{lem: surprise}
Let $(\Omega,P,\pi)$ be a finite irreducible Markov chain. Let  $f_{A}(x):=1_{x \in A}/\pi(A)$. $$\forall A \subset \Omega, \quad \max (\frac{1}{2} \|(f_{A})_{*}\|_{1},   \|f_{A}^* \|_1 )\le e \max (1 , |\log \pi(A)| ). $$
\end{lemma}
\begin{proof}
We first show that $\|f_{A}^{*}\|_{1} \le  e \max (1 , |\log \pi(A)| )$. By \eqref{eq: ergodic1} for all $1<p<\infty$
$$\|f_{A}^{*}\|_{1} \le \|f_{A}^{*}\|_{p} \le p^{*} \|f_{A}\|_p =  p^{*}[\pi(A)]^{-1/p^{*}}  $$
Taking $p^{*}:=\max( 1+\epsilon ,|\log \pi(A)|)$ and sending $\epsilon$ to 0 (noting that the r.h.s.~is continuous w.r.t.~$p_*$) concludes the proof. The same calculation shows that (last inequality) \[\|(f_{A})_{*}\|_{1} \le \|(f_{A})_{*,\mathrm{even}}\|_{1}+\|(Pf_{A})_{*,\mathrm{even}}\|_{1} \le 2\|(f_{A})_{*,\mathrm{even}}\|_{1} \le 2 e \max (1 , |\log \pi(A)| ).  \] \end{proof}
We note that by \cite[Theorem 2]{starr} $(1-e^{-1})\|f^*_A\|_{1}-1 \le   \|f_A \log [\max (1,|f_A|)] \|_1=|\log \pi (A)| $.
\subsection{Bounding escape probabilities using $\kappa$ }
Recall that $P_A$ and $Q_A$ are the restriction to $A$ of $P$ and $Q$, resp.. Denote $$H_t^A(x,y):=e^{-t(I-P_{A})}(x,y)=\Pr_{x}(X_t=y,T_{A^c}>t) \quad \text{and similarly} \quad S_t^A:=e^{-t(I-Q_{A})}.$$
Recall that $\la(A)$ is the smallest eigenvalue of $I-Q_A$. By the Perron-Frobenius Theorem there exists a distribution $\mu_{A}$ on $A$, known as the {\em \textbf{quasi-stationary distribution}} of $A$, satisfying that the escape time from $A$ w.r.t.~$Q$, starting from $\mu_A$, has an Exponential (resp.~Geometric in discrete-time) distribution with mean  $\rel(A)=1/\la(A)$. Equivalently, for all $t \ge 0$      $$\mu_A Q_A= (1-\lambda(A))\mu_A \text{ and }\mu_A S_t^A= e^{-\lambda(A)t}\mu_A.$$ Throughout we use $\mu_A$ to denote the quasi-stationary distribution of $A$. Recall that we denote $\pi$ conditioned on $A$ by $\pi_A$.

\medskip

Using the spectral decomposition of $Q_A$ (e.g.~\cite[Lemma 3.8]{cutoff} or \cite[(3.87)]{aldous}) it follows that
\begin{equation}
\label{eq: stationary<quasi}
\begin{split}
\forall A \subsetneq \Omega,\, s \ge 0, \quad & \Prr_{\pi_{A}}[T_{A^c}>s] \le  \Prr_{\mu_A}[T_{A^c}>s] =\mu_A S_t^A 1_A = e^{-\la(A) s}\mu_A(A)=e^{-\la(A) s}.
\\ \forall A \subsetneq \Omega,\, k \ge 0, \quad & \pi_A Q_A^k 1_A \le \mu_A Q_A^k 1_A=(1-\lambda(A))^k\mu_A(A)=(1-\lambda(A))^k.
\end{split}
\end{equation}
\begin{proposition}
\label{prop: comparison}
For reversible chains
\begin{equation}
\label{eq: com1}
  \kappa \le  3\rho \quad \text{and} \quad \kappa_{\mathrm{discrete}} \le  3\rho_{\mathrm{discrete}}.
\end{equation}

\end{proposition}

{\em Proof:} We first show that $\kappa \le  3\rho $. Let $A \in \mathrm{Con}_{1/2}$ be such that  $\kappa=\rel(A) |\log \pi(A)|$. By \eqref{eq: stationary<quasi} $\Pr_{\mu_{A}}[T_{A^{c}}>\kappa/3]= \pi(A)^{1/3}$. Since $a^{1/3} > a+\frac{1}{2}\sqrt{a(1-a)}, $ for all $0 \le a \le 1/2$, we have that   $$\max_{x \in A}\Pr_{x}[T_{A^{c}}>\kappa/3] \ge  \Pr_{\mu_{A}}[T_{A^{c}}>\kappa/3]= \pi(A)^{1/3}> \pi(A)+\frac{1}{2}\sqrt{ \pi(A)\pi(A^c)}.  $$
We now show that $\kappa_{\mathrm{discrete}} \le  3\rho_{\mathrm{discrete}} $. Let $B\in \mathrm{Con}_{1/2}$ be such that  $\kappa_{\mathrm{discrete}}= |\log_{\frac{1}{1-\la(A)}} \pi(B)|$. Denote the hitting time of $B^c$ w.r.t.~the discrete time chain as $T_{B^{c}}^{\mathrm{discrete}}$. By \eqref{eq: stationary<quasi} \[\max_{x \in B}\Pr_{x}[T_{B^{c}}^{\mathrm{discrete}}>\kappa_{\mathrm{discrete}}/3]  \ge \mu_B P_B^{\kappa_{\mathrm{discrete}}/3} 1_B  =\pi(B)^{1/3}> \pi(B)+\frac{1}{2}\sqrt{ \pi(B)\pi(B^c)}. \qed  \]
%
%
%

\begin{definition}
\label{def: rhodef}
 $\bar \rho:=\max_x \bar \rho_x$ and $\bar \rho_{\entt}:=\max_x \rho_{\entt,x} $, where
\begin{equation}
\label{eq: barrhox}
\begin{split}
& \bar \rho_{x}:= \min \{t: \Pr_{x}[T_{A^c}>t] \le  \pi(A)^{3} \text{ for all } A \in \mathrm{Con}_{1/2} \}.
\\ & \bar \rho_{ \entt , x}:= \min \{t: \Pr_{x}[T_{A^c}>t] \le  \frac{ 1}{16e^{2}[\log (e^{3/2}/\pi(A)) ]^3} \text{ for all } A \in \mathrm{Con}_{1/2} \}.
\end{split}
\end{equation}
\end{definition}
Note that by the Markov property, $ \max_x \Pr_{x}[T_{A^c}>mt] \le (\max_{y} \Pr_{y}[T_{A^c}>t])^m $ and so
\begin{equation}
\label{eq: rhobarrho}
 \quad \rho \le \bar \rho \le  9 \rho \quad \text{and} \quad c' \rho_{\entt} \le \bar \rho_{\entt} \le  C' \rho_{\entt},
\end{equation}
for some absolute constants $c',C'>0$ (for $\bar \rho \le 9 \rho $, use the inequality  $a^{1/3} \ge a+\frac{1}{2}\sqrt{a(1-a)} $, valid for all $0 \le a \le 1/2$). The following proposition refines the inequality $\bar \rho \le  9 \rho$.
\begin{proposition}
\label{prop: rhovsbarrho}
For every reversible chain,
\begin{equation}
\label{eq: rhovsbarrho}
\forall x \in \Omega, \quad \bar \rho_x \le \rho_x +s, \quad \text{where} \quad s:=8 \kappa +  2\rel \log 8.
\end{equation}
\end{proposition}
\emph{Proof:} Let $x \in \Omega$ and $A \in \mathrm{Con_{1/2}} $. By \eqref{eq: laAla} $2 \rel \ge \max_{B \in \mathrm{Con}_{1/2}} \rel(B)   $ and so   by \eqref{eq: stationary<quasi}  $$\Pr_{\pi_A}[T_{A^{c}}> s ] \le e^{-\la(A)[\rel(A)(8|\log \pi(A) |+\log 8) ]}  = \pi(A)^{8}/8. $$
Thus the set $$B=B(A):=\{y: \Pr_y[T_{A^{c}} > s ]>\pi(A)^{3}/2\} $$ satisfies $$\pi(B)/ \pi(A)=  \pi_A(B) < \Pr_{\pi_A}[T_{A^{c}}> s ]/(\pi(A)^{3}/2)   \le  \pi(A)^{5}/4, $$ and so by the definition of $\rho_x$, $\Pr_x[T_{B^{c}}>\rho_x] \le \pi(B)+\frac{1}{2}\sqrt{ \pi(B)\pi(B^c) }  \le \sqrt{ \pi(B)}   \le \frac{1}{2} \pi(A)^{3}$ (where we used $ \pi(B)<2^{-8} $). Finally, by the definition of $B$ and the Markov property  $$\Pr_{x}[T_{A^c}>\rho_x+s] \le \Pr_x[T_{B^{c}}>\rho_x]+\max_{b \notin B } \Pr_b[T_{A^{c}}>s]\le  \frac{1}{2} \pi(A)^{3}+  \frac{1}{2} \pi(A)^{3}=\pi(A)^{3}. \quad \qed $$

\begin{lemma}
\label{lem: rhodiscrhocts} For every finite irreducible Markov chain we have that  \[ \rho_{{\mathrm{discete}}}  \le  C \rho,  \]  \[ \rho_{\entt}^{{\mathrm{discete}}} \le C' \rho_{\entt}.\] 
\end{lemma}
\emph{Proof:}
 Let $A \in \con_{1/2}$ and $x \in \Omega$. To avoid ambiguity we denote the distributions of the discrete and the continuous-time chains started at $x$ by $\Pr_x$ and $\h_x$, resp..   Since for all $M \in \N $ we have that $\h_x[T_{A}>Mt] \le (\max_y \h_y[T_A>t])^M$ it suffices to show that for all $t \in \N $ we have that $\Pr_{x}[T_{A}>4t] \le 4\h_{x}[T_{A}>t]   $. Indeed, if $N_t \sim \mathrm{Pois}(t)$ then \[\h_{x}[T_{A}>t]=\sum_{k}\Prr[N_t=k]\Pr_x[T_A>k] \ge \Prr[N_t \le 4t ]\Pr_x[T_A>4t] \ge \frac{1}{4}\Pr_x[T_A>4t].  \qed   \]
\subsection{Related work}
\label{s: related}
Let  \[\mathrm{hit}(\epsilon):=\max_x \mathrm{hit}_{x}(\epsilon), \quad \mathrm{hit}_{x}(\epsilon):=\min \{t: \Pr_{x}[T_{A} >t ] \le \epsilon,\, \forall A \in \con_{1/2} \} .\] 
Let $ t_{\mathrm{mix}}(\epsilon):=\tau_1(2\epsilon)$ be the total-variation $\epsilon$-mixing-time. In \cite{cutoff} it was shown that for finite irreducible reversible chains, for all $ \epsilon \in (0,1),\delta \in (0,\frac{1}{2} \min (\epsilon,1-\epsilon)) $ we have that \begin{equation}
\label{eq: hit}
 \mathrm{hit}(\epsilon + \delta)-4 \rel | \log  \delta| \le t_{\mathrm{mix}}(\epsilon) \le \mathrm{hit}(\epsilon- \delta)+4 \rel | \log \delta  |,
\end{equation}
 Generally, $t_{\mathrm{rel}}|\log (2\epsilon)| \le t_{\mathrm{mix}}( \epsilon)$ for all $0<\epsilon \le 1/2$, however often $\rel \ll \tau_1$. In particular, this is the case for a sequence of reversible chains which exhibits cutoff (i.e.~abrupt convergence) in total variation (\cite[Lemma 18.4]{levin}). In \cite[Theorem 3]{cutoff} \eqref{eq: hit} is exploited in order to obtain a characterization of the cutoff phenomenon for reversible Markov chains, in terms of concentration of hitting times of ``worst" (in some sense) sets.

\medskip

The main tool in the proof of \eqref{eq: hit} is Starr's $L_p$ maximal inequality (Theorem \ref{thm: maxergodic}). In other words, an $L_p$ maximal inequality is used to characterize convergence in $L_1$. A look into the proof of \eqref{eq: hit} reveals that it does not require the full strength of Starr's inequality. It is thus natural to try applying Starr's $L_p$ maximal inequality to study stronger notions of convergence. Indeed Theorem \ref{thm: mainintro} can be seen as the $p>1$ counterpart of \eqref{eq: hit}. Also in our analysis the main tool is Starr's inequality.

\section{An overview of our approach}
\label{s: ideas}
\subsection{Lower bounding mixing times using hitting times}
\label{s:ideaslower}
We start with an illustrating example:  if $\Pr_x[T_{A^c}>t]>3\pi(A)/2 $ for some set $A$, then $$[H_t(x,A)-\pi(A) ]/\pi(A) \ge [\Pr_x[T_{A^c}>t]-\pi(A)]/\pi(A) >1/2.$$ Denote $\pi$ conditioned on $A$ by $\pi_A(a):=1_{a \in A}\pi(a)/\pi(A)$. Finally, note that
\[ d_{\infty,x}(t) \ge \max_{a \in A}h_t(x,a)-1 \ge \sum\pi_A(a)(h_t(x,a)-1)=[H_t(x,A)-\pi(A) ]/\pi(A) >1/2.  \]
Hence $\tau_{\infty,x} \ge \min \{t: \Pr_x[T_{A^c}>t] \le 3\pi(A)/2, \text{ for all }A \} $.

\medskip

This generalizes as follows. Let $\PP(\Omega)$ be the collection of all distributions on $\Omega $. Let $A  \subsetneq \Omega $, $x \in \Omega$, $t>0$ and $\delta \in (0,1) $.  Let \[\PP_{A,\delta}:=\{\mu \in \PP(\Omega): \mu(A) \ge \pi(A)+\delta \pi (A^c) \}.\] Clearly, if $\Pr_x[T_{A^c}>t] \ge   \pi(A)+\delta \pi (A^c) $, then $\Pr_x^t \in \PP_{A,\delta} $. Note that \[\nu_{A,\delta}:= \delta \pi_A+(1-\delta)\pi \in \PP_{A,\delta} . \]
Moreover, $\min \{\delta':\nu_{A,\delta'} \in \PP_{A,\delta}   \}=\delta $. It is thus intuitive that for a convex distance function between distributions, $\nu_{A,\delta} $ is the closest distribution to $\pi$ in $\PP_{A,\delta}$.
\begin{proposition}
\label{prop: Lagrange}
Let $(\Omega,P,\pi)$ be some finite irreducible Markov chain. Let $A  \subsetneq \Omega $. Denote $\nu_{A,\delta}:= \delta \pi_A+(1-\delta)\pi$.  Then for all $ \delta \in (0,1)$,
\begin{equation}
\label{eq: Lagrange}
\begin{split}
&  \quad  \min_{\mu \in \PP_{A,\delta}}\| \mu-\pi\|_{2,\pi}=\|\nu_{A,\delta}-\pi\|_{2,\pi}=\delta \sqrt{\pi(A^c)/\pi(A)}. \quad \\ &  \quad \min_{\mu \in \PP_{A,\delta}}D(\mu \| \pi)=D(\nu_{A,\delta} \| \pi)  =u(\pi(A),\delta),
\end{split}
\end{equation}
where $u(x,y):=[ y+x(1-y)]\log (1+\frac{y(1-x)}{x} )+(1-y)(1-x)\log (1-y) $.
\end{proposition}
\begin{proof}
The first equality in both lines can be verified using Lagrange multipliers. The second equality in both lines is straightforward.\end{proof}
Proposition \ref{prop: Lagrange} motivates the definitions in \eqref{eq: rhosubx}. We argue that \eqref{eq: Lagrange}, implies the first inequalities in both \eqref{eq: main3intro}-\eqref{eq: main3'intro} by making suitable substitutes for $\delta$ in \eqref{eq: Lagrange}.   For  \eqref{eq: main3intro} substitute $\delta=\frac{1}{2}\sqrt{\frac{\pi(A)}{\pi(A^c)}} $ in the first line of \eqref{eq: Lagrange}. For every  $x \in \Omega$ and $t<\rho_x $ there is some  $A \in \con_{1/2}$ such that \[\Pr_{x}[T_{A^c}>t] > \pi(A)+\frac{1}{2}\sqrt{\pi(A)\pi(A^c)}=\pi(A)+\delta \pi (A^c), \]  where the equality follows by our choice of $\delta$. As mentioned above, this implies that $\Pr_x^t \in  \PP_{A,\delta'}  $ for some $\delta'>\delta$ and so by \eqref{eq: Lagrange} and the choice of $\delta$, we have that  $\| \Pr_x^t -\pi\|_{2,\pi} > 1/2 $. For \eqref{eq: main3'intro}, it is not hard to verify that for some $C',C_{\entt}>0$, we have that $u(x,\min( \frac{C'}{  | \log x| },\frac{\frac{99}{100}-x)}{1-x}) \ge 1/2 $ and $x+\frac{C'}{  | \log x| } (1-x) \le \frac{C_{\entt}}{  | \log x| } $ for all $x \in (0, 1/2] $. Substituting $\delta=\min( \frac{C'}{  | \log x| },\frac{\frac{99}{100}-x)}{1-x})$ in the second line of \eqref{eq: Lagrange} implies the first inequality of \eqref{eq: main3'intro} in a similar manner to the above derivation of the first inequality of \eqref{eq: main3intro}.

\subsection{Upper bounding mixing times using hitting times}
\label{s:ideasupper}

We now explain the idea behind the proof of the upper bound on $\tau_2$ from \eqref{eq: main3intro}. Let $x \in \Omega$. Denote $t:= \rho_x + 8\kappa +6 \rel \log 2 $. By Theorem \ref{thm: kappa=tell} it suffices to bound $d_{2,x}(t)$.
\begin{itemize}

\item[Step 1:]  Show that (Proposition \ref{prop: rhovsbarrho})  $$\forall B \in \con_{1/2} , \quad \Pr_x[T_{B^c}>t] \le \pi(B)^3. $$

\item[Step 2:] Show that (Lemma \ref{lem: L2calculation}) for $A_{s}:=\{y:h_t(x,y) \ge (s+1) \}$ $$ \forall M \ge 1  \quad \|\Pr_x^{t}-\pi \|_{2,\pi}^2 \le M^{2}+ \int_{M}^{\infty}2s \pi(A_{s})ds. $$
$\Longrightarrow$ By Poincar\'e ineq.~\eqref{lem: contraction} it suffices that $s\pi(A_s) \le 2 s^{-3/2} $ for $s \ge M $ (for some $M$).

\item[Step 3:] For~$B_s=\{y:\sup_k H_k(y,A_s) > \frac{s}{2} \pi(A_s) \}$ by step 1 and the Markov property,
\begin{equation}
\label{eq: overview1}
\begin{split}
& s \pi(A_s) \le H_t(x,A_{s})=\Pr_x[T_{B_{s}^{c}}>t,X_{t} \in A_s]+\Pr_x[T_{B_{s}^{c}} \le t,X_{t} \in A_s]
 \\  & \le \Pr_x[T_{B_{s}^{c}}>t]+\sup_{y \notin B_s,k \ge 0}H_k(y,A_s)  \le \pi(B_s)^3+\frac{s}{2} \pi(A_s) .
 \end{split}
 \end{equation}
\item[Step 4:] If  $\pi(B_s) \le s^{-1/2}$, then we are done. Unfortunately, we do not know how to prove this estimate. Hence we have to define the set $B_s$ in a slightly different manner: $B_s:=\{y: \sup_k H_k(y,A_s)> e  \sqrt{s}|\log  \pi(A_{s}) |\pi(A_s)   \}$. By Lemma \ref{lem: surprise} indeed $\pi(B_s) \le s^{-1/2}$. Since $e  \sqrt{s}|\log  \pi(A_{s}) |\pi(A_s) \le  s \pi(A_s)/2    $, unless $\pi(A_s) \le C e^{-\sqrt{s}} $, repeating the reasoning in \eqref{eq: overview1} with the new choice of $B_s$ concludes the proof. \qed
\end{itemize}

The proof of Theorem \ref{thm: kappa=tell} is similar. The general scheme is as follows. Define a relevant family of sets $A_s$. Define $B_s$ to be of the following form $\{y:\sup |g_s(y)|>a_s \}$ with appropriate choices of $g_s$ and $a_s \in \R_+$ so that the desired inequality we wish to establish for $A_s$ holds with some room to spare given that $T_{B_s^c} \le t $  (for an appropriate choice of $t$). Finally, control the error term $\Pr[T_{B_s^c} > t] $ (using the choice of $t$) by controlling $\pi(B_s)$ using an appropriate maximal inequality.
\section{Proofs of the main results}
\subsection{An upper bound on $\tau_2 $}
In this section we prove the following theorem, which refines \eqref{eq: main3intro} from Theorem \ref{thm: mainintro}.

\begin{maintheorem}
\label{thm: main}
For every finite irreducible reversible Markov chain $(\Omega,P,\pi)$  we have that
\begin{equation}
\label{eq: main1}
\forall x, \quad \rho_{x} \le \tau_{2,x}  \le \bar \rho_{x}+5\rel \le \rho_{x}+8 \kappa+(5+6 \log 2 )\rel  .
\end{equation}

The same holds when $x$ is omitted from all of the terms above. Consequently,
\begin{equation}
\label{eq: main3}
 \quad \rho \le \tau_{2}  \le (9 +15 / \log 2) \rho.
\end{equation}
\end{maintheorem}

\begin{lemma}
\label{lem: L2calculation} Let $A_{x,t}(s):=\{y:h_t(x,y) \ge s+1 \}$. For every finite irreducible reversible chain, for all  $x \in \Omega$  and $\ell \ge 1$
$$\forall t \ge 0 ,\,  \quad \|\Pr_x^{t}-\pi \|_{2,\pi}^2 \le \ell^2 + \int_{\ell}^{\infty}2s^{} \pi(A_{x,t}(s))ds.$$
\end{lemma}
{\em Proof:}  Fix some $x \in \Omega$, $t \ge 0$ and $\ell \ge1$. Let $f(y):=|h_t(x,y)-1|$. Then $\|\Pr_x^{t}-\pi \|_{2,\pi}^2=\|f\|_2^{2}=\mathbb{E}_\pi[f^2]$. Note that for all $s > 1$, $\{f \ge s \}=A_{x,t}(s)$. Observe that $$\mathbb{E}_\pi[f^21_{f > \ell}]=\int_{0}^{\infty}2s \pi(\{f1_{f>\ell}>s \})ds \le \pi(f>\ell)\ell^2  + \int_{\ell}^{\infty}2s \pi(A_{x,t}(s))ds.$$
Finally, since $f^2\le f^21_{f > \ell}+1_{f \le \ell}\ell^2 $, we get that
$$\mathbb{E}_\pi[f^2] \le \pi(f \le  \ell)\ell^2   +\mathbb{E}_\pi[f^21_{f > \ell}] \le \ell^2+\int_{\ell}^{\infty}2s \pi(A_{x,t}(s))ds. \quad \qed$$

\emph{Proof of Theorem \ref{thm: main}:} We first note that \eqref{eq: main3} follows from \eqref{eq: main1} in conjunction with \eqref{eq: com1} and \eqref{eq: laAla}. We now prove \eqref{eq: main1}. Let $x \in \Omega$. The inequality $\rho_{x} \le \tau_{2,x}$ follows from \eqref{eq: Lagrange} and the discussion following it.   Set $t:=\bar \rho_x $. As above, denote $A_{s}:=\{y:h_t(x,y) \ge s+1 \}$. By the Poincar\'e inequality (Fact \ref{lem: contraction}) and Lemma \ref{lem: L2calculation} it suffices to show that  $$ \int_{e^{4}}^{\infty}2s \pi(A_{s})ds \le e^{7} \le e^{10}/4-e^{8}. $$

Let $g_s(y):=\sup_k H_k(y,A_{s})/\pi(A_{s})$. By Lemma \ref{lem: surprise} $\|g_{s}\|_1 \le e|\log  \pi(A_{s})|  $.  Let  $$B_s:=\{y: g_s(y) >2e^{-1}  \sqrt{s+1}|\log  \pi(A_{s}) |  \}=\{y: \sup_k H_k(y,A_{s})\ge2e^{-1}  \sqrt{s+1}\pi(A_{s})|\log  \pi(A_{s}) |  \}.$$ Let $s \ge e^{4}$. By Markov inequality $\pi(B_s) \le e^{2}/(2\sqrt{s+1} )\le \frac{1}{2} $ and so by the definition of $\bar \rho_x$    $$\Pr_x[T_{B_s^{c}}>t,X_{t} \in A_s] \le \Pr_x[T_{B_s^{c}}>t] \le  \frac{e^{6}}{8(s+1)^{3/2}  }. $$
 Also, by the definition of $B_s$ we clearly have that
$$\Pr_x[T_{B_s^{c}} \le t,X_{t} \in A_s] \le \sup_{b \notin B_s,k \ge 0}H_k(b,A_s) \le 2e^{-1}  \sqrt{s+1}\pi(A_s)|\log  \pi(A_{s}) | .  $$
Since by the definition of $A_s$ (first inequality) $$(s+1) \pi(A_s) \le H_t(x,A_{s})=\Pr_x[T_{B_s^c}>t,X_{t} \in A_s]+\Pr_x[T_{B_s^c} \le t,X_{t} \in A_s],$$ we get that if $\Pr_x[T_{B_s^{c}}>t,X_{t} \in A_s] \le \Pr_x[T_{B_s^{c}} \le t,X_{t} \in A_s]$, then $$(s+1)\pi(A_s) \le 4e^{-1} \sqrt{s+1}\pi(A_s)^{}|\log  \pi(A_{s}) | ,$$
 which simplifies as follows
$$2s\pi(A_s) \le 2se^{-e\sqrt{s+1}/4 } . $$
while if $\Pr_x[T_{B_s^{c}}>t,X_{t} \in A_s] > \Pr_x[T_{B_s^{c}} \le t,X_{t} \in A_s]$, then we have that $$2s \pi(A_s) < 4\Pr_x[T_{B_s^{c}} > t,X_{t} \in A_s] \le \frac{e^{6}}{2(s+1)^{3/2}   }. $$ Let $f(s)=15(s+1)^{3/2}e^{-e\sqrt{s+1}/4 } $. Then $f(e^4)< e^{6} $ and for $s \ge e^{4} $ we have that $ \frac{d}{ds}(-f(s)) \ge( \frac{15e}{8}-\frac{45}{2e^2})(s+1)e^{-e\sqrt{s+1}/4 } \ge 2se^{-e\sqrt{s+1}/4 }   $. Hence indeed
$$ \int_{e^{4}}^{\infty} 2s \pi(A_s)ds \le \int_{e^{4}}^{\infty} \max (2se^{-\sqrt{s+1}/(2e) } , \frac{e^{6}}{2(s+1)^{3/2}   } ) ds \le f(e^4)+ \frac{e^{6}}{(e^{4}+1)^{\frac{1}{2}} } \le e^{7}. \qed  $$

\subsection{A hitting times characterization of mixing in relative entropy}
Recall the definitions of $\rho_{\entt} , \bar \rho_{\entt}, \rho_{\entt,x}$  and $\bar \rho_{\entt,x}$ from \eqref{eq: rhorhoent} and \eqref{eq: barrhox}. Recall that by \eqref{eq: rhobarrho}, $c\rho_{\entt} \le \bar \rho_{\entt} \le C \rho_{\entt} $.
The following theorem refines \eqref{eq: main3'intro} from Theorem \ref{thm: mainintro}.
\begin{maintheorem}
\label{thm: main3}
Let $(\Omega,P,\pi)$ be a finite irreducible reversible Markov chain. Then
\begin{equation}
\label{eq: main31}
\forall x, \quad \rho_{x,\entt}\le \tau_{\entt,x}  \le \bar \rho_{x,\entt}+14\rel .
\end{equation}
The same holds when $x$ is omitted from all of the terms above. Consequently
\begin{equation}
\label{eq: main33}
 \rho_{\entt}  \le \tau_{\entt} \le C_{1} \rho_{\entt} .
\end{equation}
\end{maintheorem}
\emph{Proof of Theorem \ref{thm: main3}:} Let $x \in \Omega$.  The inequality $\rho_{x,\entt}\le \tau_{\entt,x}   $ follows from \eqref{eq: Lagrange} and the discussion following \eqref{eq: Lagrange}. The inequality $\tau_{\entt} \le C_{1} \rho_{\entt}$ follows from \eqref{eq: main31} and \eqref{eq: rhobarrho}, in conjunction with the fact that (under reversibility) $c \rel \le \rho_{\entt}$ for some absolute constant $c>0$ (c.f.~\cite[(3.19)]{cutoff} for the fact that there exist some $A \in \con_{1/2}$ and $a \in A$ so that $\Pr_{a}[T_{A^{c}}>\epsilon \rel ] \ge e^{-\epsilon} \ge 1-\epsilon $, for all $\epsilon \ge 0$).
We now prove that $\tau_{\entt,x} \le \bar \rho_{x,\entt} + 14 \rel $. Denote $r:= \bar \rho_{x,\entt} $,  $r':=14 \rel$.  Let $$D:=\{y:h_r(x,y)>e^{10} \}.$$ Denote $\delta:=H_r(x,D)-e^{10} \pi(D) $, $$\mu(y):=\delta^{-1}1_{y \in D}[H_{r}(x,y)-e^{10} \pi(y)],$$
$$\nu(y):=(1-\delta)^{-1}[1_{y \notin D}H_{r}(x,y)+1_{y \in D}e^{10} \pi(y)]. $$
Denote $\mu_\ell:= \mu H_{\ell} $ and  $\nu_\ell:= \nu H_{\ell} $.
Then $\Pr_x^{r+r'}= \delta \mu_{r'}+(1-\delta)\nu_{r'}$ and so by convexity (which holds for $D(\cdot \| \pi)$ by Jensen's inequality applied to each $y$ separately) and \eqref{eq: su}
\begin{equation}
\label{eq: triangleineq}
D(\Pr_x^{r+r'} || \pi) \le \delta D(\mu_{r'} || \pi)+(1-\delta) D(\nu_{r'} || \pi) \le \delta D(\mu_{r'} || \pi)+(1-\delta)\log (1+\|\nu_{r'}-\pi\|_{2,\pi}^2). \end{equation}
By \eqref{eq: L2contraction} $$\|\nu_{r'}-\pi\|_{2,\pi} \le \|\nu-\pi\|_{2,\pi}e^{-14} \le \|\nu-\pi\|_{\infty,\pi}e^{-14}   \le (1-\delta)^{-1}e^{-4}. $$ Using $\sqrt{1+a} \le 1+\sqrt{a}$ and $\log (1+a) \le a$ we get that \[ (1-\delta)\log (1+\|\nu_{r'}-\pi\|_{2,\pi}^2) \le 2(1-\delta)\log \left(1+\|\nu_{r'}-\pi\|_{2,\pi}\right) \le2  e^{-4}. \]
By \eqref{eq: triangleineq} to conclude the proof it is left to show that $ \delta D(\mu_{r'} || \pi) \le 1/2 - 2  e^{-4} $. Denote $$a_{y}:=1_{y \in D}[H_{r}(x,y)-e^{10}\pi(y)], \quad g(y)=a_{y}/\pi(y). $$  $$\delta D(\mu || \pi)=\sum a_{y}\log (g(y)/\delta )= \delta |\log \delta|+\mathbb{E}_{\pi}[g \log g ].  $$

Since $\delta |\log \delta| \le 1/e$, for all $\delta \in [0,1] $,  in order to show that  $ \delta D(\mu_{r'} || \pi) \le 1/2 - 2 e^{-4} $ \begin{equation}
\label{eq: suffices}
 \text{it suffices to show that} \quad \mathbb{E}_{\pi}[g \log g ] \le 1/10 < 1/2 - 1/e -2e^{-4} .
\end{equation}
Similarly to the proof of Theorem \ref{thm: main}, let  $$A_s=\{y:g(y) \ge s \} \text{ and } B_s:=\{y: \sup_{\ell} H_\ell (y,A_s) >  \sqrt{s+e^{10}}\pi(A_s)|\log \pi (A_s)| \} .$$ Then
\begin{equation}
\label{eq: glogg}
 \mathbb{E}_{\pi}[g \log g ] \le \int_{0}^\infty \pi(\{y:g(y)\log g(y)>s\})ds = \int_{1}^\infty (1+\log s)\pi(A_{s})ds.
\end{equation}
Note that $(e^{10}+s)\pi(y) \le H_r(x,y) $ for every $y \in A_s$. Hence as in the proof of Theorem \ref{thm: main}
\begin{equation}
\label{eq: Bsent'}
(e^{10}+s)\pi(A_{s})\le H_r(x,A_s) \le \Pr_{x}[T_{B_s^c}>r]+\mathbb{E}_x[X_r \in A_s \mid T_{B_s^c}\le r  ].
\end{equation}
By the definition of $B_s$ and the Markov property,
\begin{equation}
\label{eq: Bsent}
\mathbb{E}_x[X_r \in A_s \mid T_{B_s^c}\le r  ]\le \sup_{y \notin B_s,\ell \ge 0}H_{\ell}(y,A_s)  \le \sqrt{s+e^{10}}\pi(A_s)|\log \pi (A_s)|.
\end{equation}
 By Lemma \ref{lem: surprise} $\pi(B_s) \le e/\sqrt{s+e^{10}} \le 1/2 $ and hence by the definition of $r$, $$\Pr_{x}[T_{B_s^c}>r]\le \frac{1}{16e^{2}(\frac{1}{2}(\log (s+e^{10})+1))^{3}}=\frac{1}{2e^{2}(1+\log (s+e^{10}))^{3}}. $$
As in  the proof of Theorem \ref{thm: main}, it follows that for all $s \ge 1$, $(s+e^{10})\pi(A_{s}) \le  \frac{2}{2e^{2}(1+\log (s+e^{10}))^{3}} $, as otherwise by \eqref{eq: Bsent'} $(s+e^{10})\pi(A_{s})<2\mathbb{E}_x[X_r \in A_s \mid T_{B_s^c}\le r  ] $, which by \eqref{eq: Bsent} implies that \[\pi(A_s) \le \exp (-\frac{1}{2} \sqrt{s+e^{10}})\le \exp (- \sqrt{s/8}-\sqrt{e^{10}/8})< e^{-50- \sqrt{s/8}} < \frac{(s+e^{10})^{-1}}{e^{2}(1+\log (s+e^{10}))^{3}} ,\] a contradiction. Thus for all $s \ge 1$,
$$(1+\log s)\pi(A_{s}) \le \frac{1}{e^{2}(s+e^{10})(1+\log (s+e^{10}))^{2}}, $$
which yields that $\int_{1}^\infty (1+\log s)\pi(A_{s})ds \le  \int_{1+e^{10}}^\infty \frac{e^{-2}ds}{s (1+\log s)^2} =\frac{e^{-2}}{ 1+\log (1+e^{10})} < e^{-2}/11  $. This concludes the proof using \eqref{eq: suffices} and \eqref{eq: glogg}. \qed

\subsection{Proof of Theorem \ref{thm: kappa=tell}}

\emph{Proof of Theorem \ref{thm: kappa=tell}:} As mentioned in the introduction, it is known that  $ \kappa \le t_{\mathrm{LS}}$. Denote $r:=\frac{1}{2} \kappa $. Note that $P$ and $Q=(P+P^{*})/2$ have the same $\rel$ and $t_{\mathrm{LS}}$. Thus we may work with $S_t=e^{-t(I-Q)}$ instead of $H_t$. By \eqref{eq: hyper} it suffices to show that  $\|S_{r } \|_{2 \to 4 } \le 7$.  Fix some $f \in \R^{\Omega} $ such that $\|f \|_2=1$. Our goal is to show that $\|S_{r }f \|_4 \le 7 $. By considering $|f|$ instead of $f$ we may assume that $f \ge 0$.  Let \[A_{s}:=\{x: S_{r  }f(x) \ge s \}. \] Then $\|S_{r }f \|_4^4 =\int_{0}^{\infty}4s^3\pi(A_s)ds\le 6^4 +\int_{6}^{\infty}4s^3\pi(A_s)ds $. Hence to conclude the proof
\begin{equation}
\label{eq: sufficeshyper}
\text{it suffices to show that} \quad \int_{6}^{\infty}4s^3\pi(A_s)ds \le 256 \le 7^4 - 6^4 .
\end{equation}

Recall that $S_tf(x)=\mathbb{E}_x[f(Y_t)]$ and that for all $A \subset \Omega$,  $S_t^{A}f(a)=\mathbb{E}_a[f(Y_t)1_{T_{A^{c}}>t}]$.  Let
$$B_s:=\{x: \sup_{t} S_{t }f(x) > s/2 \}=\{f^* > s/2 \}, \quad \text{where} \quad f^*(x)=\sup_{t} S_{t }f(x)   $$
$$D_s:=\{x \in B_s : \mathbb{E}_x[f(Y_{r  })1_{T_{B_s^{c}}>r }]  \ge s/2 \},\quad  F_{s}:=\{x\in B_s  : \mathbb{E}_x[f^{2}(Y_{r  })1_{T_{B_s^{c}}>r }]  \ge s^{2}/4 \}.$$
 By the Markov property (first inclusion), $A_s \subset D_s \subset F_s $ (the second inclusion follows by the Cauchy-Schwarz inequality). Thus $\pi(A_s) \le \pi(F_s)$. Hence, by \eqref{eq: sufficeshyper} in order to conclude the proof it suffices to show that $\int_{6}^{\infty}4s^3\pi(F_s)ds \le 256 $. By Starr's maximal  inequality \eqref{eq: ergodic1} we know that $\int_{0}^{\infty}16s\pi(B_s)ds=64\|f^*\|_2^2 \le 256\|f\|_2^2=256 $. Thus in order to show that $\int_{6}^{\infty}4s^3\pi(F_s)ds \le 256 $, and conclude the proof, it suffices to show that for all $s \ge 6$ we have that $\pi(F_s) \le  4s^{-2}\pi(B_s)$.

\medskip

Fix some $s \ge 6 $.
Note that since $\|f^*\|_2^2 \le 4 $, by Markov inequality we have that $\pi(B_s) \le 16/s^2 < 1/2 $. Using the spectral decomposition of the restriction of $f $ to $B_s$ (c.f.~\cite[Lemma 3.8]{cutoff}) and the choice of $r$
$$\mathbb{E}_{\pi_{B_{s}}}[f^{2}(Y_{r  })1_{T_{B_s^{c}}>r }] \le \mathbb{E}_{\pi_{B_{s}}}[f^{2}(Y_{0  })] e^{-2 \lambda(B_{s})r } \le (\|f\|_2^2/\pi(B_s))  e^{-2 \lambda(B_{s})r } = (1/\pi(B_s)) \times \pi(B_s)=1 . $$
Thus by the def.~of $F_s$, $\frac{1}{4}s^{2}\pi_{B_{s}}(F_s) \le \sum_{y \in F_s}\pi_{B_{s}}(y)\mathbb{E}_{y}[f^{2}(Y_{r  })1_{T_{B_s^{c}}>r }]  \le  \mathbb{E}_{\pi_{B_{s}}}[f^{2}(Y_{r  })1_{T_{B_s^{c}}>r }]\le 1$ and so indeed  $\pi(F_s) \le  4s^{-2}\pi(B_s)$. \qed

\subsection{The necessary adaptations in the proofs of the results concerning the discrete-time and averaged chains}
\label{s:disc}
The proofs of the lower bounds $\rho_{{\mathrm{discete}}} \le \min( \tau_2, \tau_2^{\mathrm{ave}}) $ and $\rho_{\entt}^{{\mathrm{discete}}} \le \min( \tau_{\entt}, \tau_{\entt}^{\mathrm{ave}}) $ in Theorem \ref{thm: aveL2}  are identical to those from Theorem \ref{thm: mainintro} (namely, these are ``naive" bounds that can be proven using the same argument as in \S~\ref{s:ideaslower}). The inequalities $ \rel^{\mathrm{absolute}} \log 2 \le \tau_{2}^{\mathrm{discete}} $ and  $c\rel^{\mathrm{absolute}}\le \tau_{\entt}^{\mathrm{discrete}}   $  from  \eqref{eq: discL2} and \eqref{eq: discent} follow from \eqref{eq: lowerL1rel}-\eqref{eq: lowerRErel}, resp.. The proofs of the upper bounds require the following minor adaptations:
\begin{itemize}
\item[(i)] In the definition of the sets $A_s$ (both in the proof of the $L_2$ case and of the relative-entropy case) one needs to replace $h_t(x,y) $ with $k_t(x,y)=P^{t}(x,y)/\pi(y)$.
\item[(ii)] In the applications of Starr's inequality one has to work with the discrete-time version, and thus pick up a multiplicative factor of 2 (which is a  non-issue). Namely, when applying Lemma \ref{lem: surprise} one has to use the estimate $\|(f_{A})_{*}\|_{1}\le 2e \max (1 , |\log \pi(A)| )$, instead of $   \|f_{A}^* \|_1 \le e \max (1 , |\log \pi(A)| )$.

\item[(iii)] One has to replace the Poincar\'e inequality with the discrete and averaged analogs, to be described shortly (Lemma \ref{lem: discpoincare}).
\end{itemize}
For the sake of completeness we prove \eqref{eq:disc4} in full details, leaving \eqref{eq: avent} and \eqref{eq: discent} as exercises. 
 We first note that the inequalities \[ \tau_{2}^{\av} \le \rho_{{\mathrm{discete}}} + C_1/\alpha_{\mathrm{discrete}} \quad \text{and} \quad \tau_{2}^{\mathrm{discete}}\le \rho_{{\mathrm{discete}}} + C_1(1/\alpha_{\mathrm{discrete}}+\rel^{\mathrm{absolute}})   \] are simple consequences of \eqref{eq:disc4}. Indeed, this follows from the following extension of the Poincar\'e inequality.
\begin{lemma}
\label{lem: discpoincare}
Assuming reversibility,  for all $\mu \in \PP(\Omega)$, $M \ge 1$ and $k \in \Z_+$ we have
\begin{equation}
\label{eq: discpoincare}
 \| \mu P^k - \pi \|_{2,\pi} \le \| \mu- \pi \|_{2,\pi}e^{-\frac{k}{\rel^{\mathrm{absolute}}}},  \quad \text{thus } \tau_{2}^{\mathrm{discete}} \le \tau_{2}^{\mathrm{discete}}(M/2)+ \lceil\rel^{\mathrm{absolute}} \log M \rceil. \end{equation}
\begin{equation}
\label{eq: avepoincare}
\| \mu A_{k}- \pi \|_{2,\pi} \le\| \mu- \pi \|_{2,\pi} \max (|\la_2|^k,\frac{1}{2}|\la_{|\Omega|}|^k (1+\la_{|\Omega|}) ) \le \| \mu- \pi \|_{2,\pi} \max (e^{-k/\rel },\frac{1}{2ek}),  \end{equation}
\[\text{thus} \quad \tau_{2}^{\mathrm{ave}} \le \tau_{2}^{\mathrm{discete}}(M/2)+ \lceil \max( \rel \log M,M)\rceil. \]
\end{lemma}
\emph{Proof:}
We first prove \eqref{eq: avepoincare}. The second inequality in \eqref{eq: avepoincare} follows from elementary calculus. We now explain why the first inequality in \eqref{eq: avepoincare} holds.  Let $f_{\mu}=\frac{\mu}{\pi}$.  By reversibility $\|\mu A_{k}-\pi \|_{2,\pi}=\|A_k f_{\mu}-1 \|_{2}=\|A_k (f_{\mu}-\E_{\pi}[f_{\mu}]) \|_{2}=\frac{1}{2} \|P^k(P+I) (f_{\mu}-\E_{\pi}[f_{\mu}]) \|_{2}$.  
Consider an orthonormal basis of $\R^{\Omega}$ consisting of eigenvectors $f_1,\ldots,f_{|\Omega|}$ such that $Pf_i=\la_i f_i $ for all $i$ (where $f_1=1$ and $\la_1=1$). Denote $b_j:=\E_{\pi}[f_{\mu}f_j] $. Then, \[\|\mu A_{k}-\pi \|_{2,\pi}^{2}=\frac{1}{4}\sum_{i=2}^{|\Omega|}b_j^2 \la_i^{2k}(1+\la_i)^2  \le \max(\la_2^{2k},\frac{1}{4} \la_{|\Omega|}^{2k} (1+\la_{|\Omega|})^2)\sum_{j=2}^{|\Omega|}b_j^2 .\] Substituting $\|\mu-\pi \|_{2,\pi}^{2}=\|f_{\mu}-\E_{\pi}[f_{\mu}] \|_2^2=\sum_{j=2}^{|\Omega|}b_j^2 $ in the r.h.s.\ concludes the proof of \eqref{eq: avepoincare}. For \eqref{eq: discpoincare}, \[ \|\mu P^{k}-\pi \|_{2,\pi}^{2} =\sum_{i=2}^{|\Omega|}b_j^2 \la_i^{2k} \le \max (\la_2,\la_{|\Omega|})^{2k} \sum_{i=2}^{|\Omega|}b_j^2= \| \mu- \pi \|_{2,\pi}^2e^{-\frac{2k}{\rel^{\mathrm{absolute}}}}. \qed   \]

We now prove \eqref{eq:disc4}. Define $\bar \rho_{x}^{\mathrm{discrete}}$ and  $\bar \rho^{\mathrm{discrete}}$   in an analogous manner to the definition of $\bar \rho_{x}$ and $\bar \rho$. Then, similarly to Proposition \ref{prop: rhovsbarrho}, we have that $\bar \rho_{x}^{\mathrm{discrete}} \le \rho_{x}^{\mathrm{discrete}}+C_4\kappa_{\mathrm{discrete}} $, for all $x \in \Omega$. Thus it suffices to show that for all $x$ 
\begin{equation}
\label{eq: rhobardiscup}
\tau_{2,x}^{\mathrm{discete}}(e^9/2) \le \bar \rho_{x}^{\mathrm{discrete}}. 
\end{equation}

 Denote $t:=\bar \rho_{x}^{\mathrm{discrete}}$ and  $A_{s}^{\mathrm{d}}:=\{y:k_t(x,y) \ge s+1 \}$, where $k_t(x,y):=P^t(x,y)/\pi(y) $. Since, similarly to Lemma \ref{lem: L2calculation} we have that $\|P^t(x,\cdot)-\pi(\cdot) \|_{2,\pi}^2 \le \int_{\ell}^{\infty}2s \pi(A_{s}^{\mathrm{d}})ds+\ell^2$, for all $x \in \Omega$ and $\ell \ge 1$, it suffices to show that $\int_{e^{8}}^{\infty}2s \pi(A_{s}^{\mathrm{d}})ds \le e^{15} \le e^{18}/4-e^{16}  $.  Let $g_s^{\mathrm{d}}(y):=\sup_k P^{k}(y,A_{s}^{\mathrm{d}})/\pi(A_{s}^{\mathrm{d}})$. Similarly to Lemma \ref{lem: surprise} (using the discrete-time version of Starr inequality) $\|g_{s}^{\mathrm{d}} \|_1 \le 2e|\log  \pi(A_{s}^{\mathrm{d}})|  $.  Let  $$B_s^{\mathrm{d}}:=\{y: g_s^{\mathrm{d}}(y) >4e^{-3}  \sqrt{s+1}|\log  \pi(A_{s}^{\mathrm{d}}) |  \}=\{y: \sup_k P^{k}(y,A_{s}^{\mathrm{d}})\ge4e^{-3}  \sqrt{s+1}\pi(A_{s}^{\mathrm{d}})|\log  \pi(A_{s}^{\mathrm{d}}) |  \}.$$ Let $s \ge e^{8}$. By Markov inequality $\pi(B_s^{\mathrm{d}}) \le e^{4}/(2\sqrt{s+1} )\le \frac{1}{2} $ and so by the definition of $\bar \rho_x^{\mathrm{discrete}}$    $$\Pr_x[T_{\Omega \setminus B_s^{\mathrm{d}}}>t,X_{t} \in A_s^{\mathrm{d}}] \le \Pr_x[T_{\Omega \setminus B_s^{\mathrm{d}}}>t] \le  \frac{e^{12}}{8(s+1)^{3/2}  }. $$
 Since by the definition of $A_s^{\mathrm{d}}$ (first inequality) $$(s+1) \pi(A_s^{\mathrm{d}}) \le P^{t}(x,A_{s}^{\mathrm{d}})=\Pr_x[T_{\Omega \setminus B_s^{\mathrm{d}}}>t,X_{t} \in A_s^{\mathrm{d}}]+\Pr_x[T_{\Omega \setminus B_s^{\mathrm{d}}} \le t,X_{t} \in A_s^{\mathrm{d}}],$$ we get that if $\Pr_x[T_{\Omega \setminus B_s^{\mathrm{d}}}>t,X_{t} \in A_s] \le \Pr_x[T_{\Omega \setminus B_s^{\mathrm{d}}} \le t,X_{t} \in A_s]$, then by the Markov property and the definition of $B_s^{\mathrm{d}}$ $$(s+1)\pi(A_s^{\mathrm{d}}) \le 8e^{-3} \sqrt{s+1}\pi(A_s^{\mathrm{d}})|\log  \pi(A_{s}^{\mathrm{d}}) | ,$$
 which simplifies as follows
$$2s\pi(A_s^{\mathrm{d}}) \le 2se^{-e^{3}\sqrt{s+1}/8 } . $$
while if $\Pr_x[T_{\Omega \setminus B_s^{\mathrm{d}}}>t,X_{t} \in A_s^{\mathrm{d}}] > \Pr_x[T_{\Omega \setminus B_s^{\mathrm{d}}} \le t,X_{t} \in A_s^{\mathrm{d}}]$, then we have that $$2s \pi(A_s^{\mathrm{d}}) < 4\Pr_x[T_{\Omega \setminus B_s^{\mathrm{d}}} > t,X_{t} \in A_s^{\mathrm{d}}] \le \frac{e^{12}}{2(s+1)^{3/2}   }. $$ Let $f(s)=e(s+1)^{3/2}e^{-e^{3}\sqrt{s+1}/8 } $. Then $f(e^8)< e^{6} $ and for $s \ge e^{8} $ we have that $ \frac{d}{ds}(-f(s)) \ge( \frac{e^{4}}{16}-\frac{3}{2e^3})(s+1)e^{-e^{3}\sqrt{s+1}/8 } \ge 2se^{-e^{3}\sqrt{s+1}/8 }   $. Hence indeed
$$ \int_{e^{8}}^{\infty} 2s \pi(A_s^{\mathrm{d}})ds \le \int_{e^{8}}^{\infty} \max (2se^{-e^{3}\sqrt{s+1}/8 } , \frac{e^{12}}{2(s+1)^{3/2}   } ) ds \le f(e^8)+ \frac{e^{12}}{(e^{8}+1)^{\frac{1}{2}} } \le e^{9}.$$   
This concludes the proof of \eqref{eq: rhobardiscup} and thus of \eqref{eq:disc4}. \qed

\medskip

\emph{Proof of Proposition \ref{prop:avcts}:} Denote $a_{m}(x,y):=A_{m}(x,y)/ \pi(y) $. Then\[\|A_{m}(x,\cdot)- \pi(\cdot) \|_{2,\pi}^2=\sum_y \pi(y) (a_m(x,y)-1)^2=\sum_y \pi(y) (a_m^2(x,y)-1).  \]
By reversibility, $\sum_y \pi(y) (a_m^2(x,y)-1)=-1+\sum_y \frac{A_m(x,y)A_m(y,x)}{\pi(x)}=\frac{((\frac{I+P}{2})^2P^{2m-2})(x,x)}{\pi(x)}-1$.
Denote the eigenvalues of $P$ by $1=\la_1>\la_2 \ge \cdots \ge \la_{|\Omega|} \ge -1 $ and let $f_1,\ldots,f_{|\Omega|} $ be an orthonormal basis of $\R^{\Omega}$ such that $Pf_i=\la_i f_i$ for all $i$. Denote $r_{-}(\ell):=\sum_{i: \la_i<0 } \la_i^{\ell} f_i^2(x) $ and  $r_{+}(\ell):=\sum_{i>1: \la_i>0 } \la_i^{\ell} f_i^2(x) $. Using the spectral decomposition and $(1+x)x^{k'} \le \frac{1}{ek'} $ for $-1 \le x \le 0$, we have that \[\|A_{k+k'}(x,\cdot)- \pi(\cdot) \|_{2,\pi}^2= \sum_{i=2}^{|\Omega|} \la_i^{2(k+k'-1)}\left(\frac{1+\la_i}{2}\right)^2 f_i^2(x) \] \[ \le \left(\frac{1}{2ek'} \right)^2 r_{-}(2k-2)+r_{+}(2k+2k'-2) \le  \left(\frac{1}{2ek'} \right)^2 r_{-}(2k-2)+ \la_2^{2k'+2} \|\Pr_x^{k-2}-\pi \|_{2,\pi}^2,\] where we have used $r_{+}(2k+2k'-2) \le \la_2^{2k'+2}r_{+}(2k-4)$ and (using \eqref{eq: generalLp} and $1-x \le e^{-x} $) \[r_{+}(2k-4) \le     \sum_{i=2 }^{ |\Omega|} \exp[-(2k-4)(1-\la_i)]f_i^2(x)=h_{2k-4}(x,x)-1=\|\Pr_x^{k-2} - \pi \|_{2,\pi}^2.\] Denote $k_{\ell}(x,y)=P^{\ell }(x,y)/\pi(y)$. Then
 \[0 \le k_{2 \ell +1}(x,x)=  \sum_{i=1}^{|\Omega|} \la_i^{2\ell+1}f_i^2(x). \]
Thus $r_{-}(2\ell+2)\le   \sum_{i: \la_i>0} \la_i^{2\ell}f_i^2(x) \le    \sum_{i=1 }^{ |\Omega|} e^{-2\ell(1-\la_i)}f_i^2(x)=h_{2\ell}(x,x)=\|\Pr_x^{\ell} - \pi \|_{2,\pi}^2+1  $. Hence \[\|A_{k+k'}(x,\cdot)- \pi(\cdot) \|_{2,\pi}^2 \le \left(\frac{1}{2ek'} \right)^2 (\|\Pr_x^{k-2} - \pi \|_{2,\pi}^2+1)+ \la_2^{2k'+2} \|\Pr_x^{k-2}-\pi \|_{2,\pi}^2. \qed  \]
\section{Application to robustness of mixing}
\subsection{Proof of Corollary \ref{cor:lazinessrobustness}}
\label{s:rob}
\begin{proof}
We only prove \eqref{eq: laziness} as the proof of \eqref{eq: laziness2} is analogous. It is not hard to verify that Theorem \ref{thm: mainintro} is still valid in the above setup (this can be formally deduced from Theorem \ref{thm: mainintro} via the representation of the generator appearing in the paragraph following Corollary \ref{cor:lazinessrobustness}). Hence it suffices to verify that \eqref{eq: laziness} is valid if we replace $\tau_2 $ and $\tilde \tau_2 $ by $\rho $ and $\tilde \rho$, resp.~(where $\tilde \rho$ is the parameter $\rho$ of the chain $(\tilde X_t)$). Denote by $\pi$, $\rel$ and $G$ (resp.~$\tilde \pi$, $\widetilde \rel$ and $\tilde G$) the stationary distribution, relaxation time and generator of $(X_t)$ (resp.~$(\tilde X_t)$), where $\rel $ is the inverse of the smallest non-zero eigenvalue of $-G$ (equivalently, we can write $G=K(P-I)$ for some transition matrix $P$ and $K >0$, and define $\rel:=\rel(P)/K$, where $\rel(P)$ is the relaxation time of $P$). In the notation of Corollary \ref{cor:lazinessrobustness} we have that $\tilde G(x,y)=r_xG(x,y) $ for all $x,y$ and so $\tilde \pi (x)=\frac{\pi(x)/r_x}{L}$ for all $x$, where $L:= \sum_y \pi(y)/r_y $ and $\max_y \max( r_y,1/r_y) \le M$. Hence $M^{-2} \le \pi(x)/\tilde \pi(x) \le M^{2} $ for all $x$. It follows from the extremal characterization of the relaxation time that $\rel \le M \widetilde \rel$. Indeed, using $\Var_{\pi}f:=\mathbb{E}_{\pi}[(f-\mathbb{E}_{\pi}f)^2] \le \mathbb{E}_{\pi}[(f-\mathbb{E}_{\tilde \pi}f)^2] \le LM \mathbb{E}_{\tilde \pi}[(f-\mathbb{E}_{\tilde \pi}f)^2]=LM \Var_{\tilde \pi}f$, we get that $$\rel = \max\frac{\Var_{\pi}f}{\E_{\pi}[(-Gf)f]} =\max\frac{\Var_{\pi}f}{L\E_{\tilde \pi}[(-\tilde Gf)f]} \le \max\frac{M\Var_{\tilde \pi}f}{\E_{\tilde \pi}[(-\tilde Gf)f]} = M \widetilde \rel,$$ where the maxima are taken over all non-constant $f \in \R^{\Omega}$.

Recall that if $Z \sim \mathrm{Exp}(1)$, then $\alpha^{-1} Z \sim \mathrm{Exp}(\alpha) $ for all $\alpha>0$. Let $Z_1,Z_2,\ldots$ be i.i.d.~$\mathrm{Exp}(1)$.
 A straightforward coupling of the chains in which they follow the same trajectory (i.e.~they make the same sequence of jumps, possibly at different times) in which if the $(k+1)$th jump is from vertex $x_{k}$ then the time spent at $x_{k}$ by the chains between their $k$th and  $(k+1)$th  jumps  is $\frac{Z_k}{-G(x_{k},x_k)} $ and $\frac{Z_k}{-r_xG(x_{k},x_k)} $, resp.,   shows that for all $x$ and $A$ the hitting time of $A$ starting from $x$ for the two chains, $T_{A} $ and $\tilde T_{A}$, resp., satisfy that \begin{equation}
\label{eq: std} \tilde T_A/M \leq_{\mathrm{st}} T_A \leq_{\mathrm{st}} M \tilde T_A,
\end{equation}
 where $ \leq_{\mathrm{st}} $ denotes stochastic domination.

Let $\rho_{ \le \delta,\delta'}:=\inf \{t: \max_{x,A:\pi(A) \le \delta)}\Pr_x[T_A>t] \le \delta'[ \pi(A)+\frac{1}{2}\sqrt{\pi(A) \pi(A^c)}] \}$ and similarly  $\tilde \rho_{ \le \delta,\delta'}:=\inf \{t: \max_{x,A: \tilde \pi(A) \le \delta}\Pr_x[\tilde T_A>t] \le \delta' [\tilde \pi(A)+\frac{1}{2}\sqrt{\tilde \pi(A) \tilde \pi(A^c)}] \}$. By the submultiplicity property of hitting times \[ \forall t \ge 0,m \in \Z_+ \text{ and }A \subset \Omega, \quad \max_{x} \Pr_x[\tilde T_{A}>tm] \le (\max_{x} \Pr_x[\tilde T_{A}>t])^{m},  \]
we get that for all $\delta' \in (0,1/2] $
\begin{equation}
\label{rhodeltadelta'1}
 \tilde \rho_{ \le \frac{1}{2},\delta'} \le C_0 \tilde \rho |\log \delta'|.
\end{equation}    
Similar reasoning as in the proof of \eqref{eq: rhovsbarrho} yields that (c.f.~\cite[Corollary 3.4]{cutoff})
\begin{equation}
\label{rhodeltadelta'2}
\rho \le  \rho_{ \le \frac{1}{2M^2},\frac{1}{2}}+C_{1} \rel \log M.
\end{equation}     
 Using  \eqref{eq: std} and the fact that for all $A$ we have that $\tilde \pi(A)/M^{2} \le \pi(A) \le M^{2} \tilde \pi(A) $ (first inequality) and \eqref{rhodeltadelta'1} (second inequality) we get that
\begin{equation*}
\label{rhodeltadelta'3}
\rho_{ \le \frac{1}{2M^2},\frac{1}{2}} \le M \tilde \rho_{ \le \frac{1}{2},\frac{1}{2M^{2}}} \le C_2 ( M \log M)\tilde \rho.
\end{equation*}
This, in conjunction with \eqref{rhodeltadelta'2} and $\rel \le M \widetilde \rel \le C_3 M \tilde \rho $ yields that $\rho \le C ( M \log M)\tilde \rho $, as desired. By symmetry, we also have that $\tilde \rho \le C(M \log M) \rho  $.
\end{proof}

\subsection{Robustness of trees}
\label{s: trees}
We start with a few definitions. Given a network $(V, E,(c_{e})_{e
\in E})$, where each edge $\{u,v\} \in E$ is endowed with a conductance (weight) $c_{u, v}=c_{v,u}>0$,  a  random walk on $(V, E,(c_{e})_{e
\in E})$ repeatedly does the following: when the current state is $v\in V$,  the random walk will move to vertex $u$ (such that $\{u,v\} \in E$) with probability $c_{u, v}/c_{v}$, where $c_v:=\sum_{w:\{v,w\} \in E} c_{v, w}$. This is a reversible Markov chain whose stationary distribution is given by $\pi(x):=c_x/c_V$, where $c_{V}:=\sum_{v \in V}c_v=2\sum_{e \in E}c_e$. Conversely, every reversible Markov chain can be presented in this manner by setting $c_{x,y}=\pi(x)P(x,y) $ (e.g.~\cite[Section 9.1]{levin}).

\medskip

Let $\mathcal{T}:=(V,E)$ be a finite tree. By Kolmogorov's cycle condition every Markov chain on $\cT $ (i.e.~$P(x,y)>0$ iff $\{x,y\} \in E $) is reversible. Hence we may assume that $\cT$ is equipped with edge weights $(c_e)_{e \in E}$.
Following \cite{peres2015mixing}, we call a vertex $v \in V$ a central-vertex if each connected component of $\cT \setminus \{v\}$ has stationary probability at most 1/2. A central-vertex always exists (and there may be at most two central-vertices). Throughout, we fix a central-vertex $o$ and call it the \emph{root} of the tree.
The root induces a partial order $\prec $ on $V$,  as follows. For every $u \in V$, we denote the shortest path between $u$ and $o$ by $\ell(u)=(u_0=u,u_1,\ldots,u_k=o)$. We call $u_1$ the \emph{parent} of $u$.  We say that $u' \prec u$ if $u' \in \ell(u)$ (i.e.~$u$ is a descendant of $u'$ or $u=u'$). The \emph{induced tree} at $u$ is $\cT_{u}:=\{v: u \in \ell(v) \}=\{u \} \cup \{v: v \text{ is a descendant of }u \} $. Fix some leaf $x$ and $\delta \in (0,1/2)$. Let $W_{x,\delta}$ be the collection of all $y \prec x $ such that $\pi(\cT_y) \ge \delta$ and let
\[x_{\delta}:=\mathrm{argmin} \{\pi(\cT_y): y \in W_{x,\delta} \} \]
(i.e.~$d(x,x_{\delta})=\min_{y \in W_{x,\delta} }d(x,y)  $, where $d$ denotes the graph distance w.r.t.~$\cT$).  Recall that $\alpha(A)=\la(A)/|\log \pi(A)|$ and that by Theorem \ref{thm: kappa=tell}, $\alpha:=\sup_{A \in \con_{1/2}}\alpha(A) \ge c_{\mathrm{LS}} $. Let  $D_{\beta}=D_{\beta,x}$ be the connected component of $x$ in $\cT \setminus \{x_{\beta} \}$. For a leaf $x$ we denote \[ \alpha_x(\delta) :=\alpha(D_{\delta}) \quad  \text{and} \quad \alpha_x:=\max_{\delta \in (0,1/4]}\alpha_x(\delta) \ge \alpha. \]
Let us now describe the skeleton of the argument in the proof of Theorem \ref{trees}.
\begin{itemize}
\item[Step 1:] Show that it suffices to consider leafs as initial states. More precisely
\begin{lemma}
\label{lem: leafisworst}
There exists an absolute constant $C>0$ so that if  $y \prec x $ then
\begin{equation}
\label{eq: leafisworst}
  \tau_{2,y} \le \tau_{2,x}+C (t_{\mathrm{LS}}+\sqrt{\rel \tau_1}).
\end{equation}
\end{lemma}
 \item[Step 2:] Show that for a leaf $x$ we can replace (in \eqref{eq: rhobarrho}) $\bar \rho_x$  (defined in \eqref{eq: barrhox}) with \[b_x:=\sup_{\delta \in (0,1/4]}  b_x(\delta) \quad \text{where} \quad b_x(\delta):= \min \{t:\Pr_x[T_{x_{\delta}}>t]\le \delta^3/4 \}. \]
\begin{proposition}
\label{prop: Tdelta}
Let $x$ be a leaf. Let $0<\delta  \le 1/4 $ and $A \in \con_{\delta} $. Denote $\bar A=A^{c} \setminus D_{\delta}$, where  $D_{\beta}=D_{\beta,x}$ is the connected component of $x$ in $\cT \setminus \{x_{\beta} \}$. Then
\begin{equation}
\label{eq: Tdelta}
 \Pr_x[T_{A^{c}}>b_x+3\kappa +10 \rel  ] \le \Pr_x[T_{x_{\delta}}>b_x]+\Pr_{x_{\delta}}[T_{\bar A}>3\kappa +10 \rel  ]<\delta^3/2.
\end{equation}
\end{proposition}
\item[Step 3]  For a leaf $x$ and $\delta \in (0,1/4] $, derive a large deviation estimate for $T_{x_{\delta}}$:
\begin{proposition}
\label{prop: LD}
There exists some $C>0$ so that for a leaf $x$ and $\delta \in (0,1/4] $, \begin{equation}
\label{eq: LD1}
b_x(\delta) \le \mathbb{E}_x[T_{x_{\delta}}]+ \max \left(\frac{32}{\alpha_{_{x}}(\delta)},8 \sqrt{\mathbb{E}_x[T_{x_{\delta}}]/\alpha_{_{x}}(\delta)} \right)  \le \tau_{1}+C\max( \kappa, \sqrt{\kappa \tau_{1}})  .
\end{equation}
\end{proposition}
The second inequality follows from the first using the fact that
 $\mathbb{E}_x[T_{x_{\delta}}] \le \tau_1 + C_5 \sqrt{\tau_1 \rel}$ \cite[Corollary 5.5]{cutoff}.
\item[Step 4] Similar reasoning as in the proof of \eqref{eq: rhovsbarrho} yields that (c.f.~\cite[Corollary 3.4]{cutoff}) \[\bar \rho_x \le \min \{t: \Pr_{x}[T_{A^c}>t] \le  \pi(A)^{3}/2 \text{ for all } A \in \mathrm{Con}_{1/4} \} +10 \rel \] By \eqref{eq: leafisworst}-\eqref{eq: LD1} in conjunction with \eqref{eq: main1} and \eqref{eq: kappa=tell} we have that
\end{itemize}
\begin{equation*}
\begin{split}
& \tau_2-C_1\sqrt{\rel \tau_1} \le \max_{x:x \text{ a leaf}} \tau_{2,x}+C_1t_{\mathrm{LS}} \le \max_{x:x \text{ a leaf}} \bar \rho_{x}+C_2t_{\mathrm{LS}} \\ & \le \max_{x:x \text{ a leaf}}  b_{x}+C_3t_{\mathrm{LS}} \le \tau_{1}+C_{4} \max( t_{\mathrm{LS}} , \sqrt{ t_{\mathrm{LS}} \tau_{1}}). \qed
\end{split}
\end{equation*}
\begin{remark}
While it is intuitive that ``typically" the worst initial state is a leaf (i.e.~$\tau_2=\tau_{2,x} $ for some leaf $x$), it  is not clear if this is always the case. 
\end{remark}
To conclude the proof of Theorem \ref{trees} we now prove Lemma \ref{lem: leafisworst} and Propositions \ref{prop: Tdelta}-\ref{prop: LD}.

\medskip

{\em Proof of Lemma \ref{lem: leafisworst}:}  Let $y \prec x $. Let $s:= \tau_{2,y}-M t_{\mathrm{LS}}$ for some constant $M>0$ to be determined later. We may assume $s>64 \sqrt{\rel \tau_1}$ as otherwise there is nothing to prove. By \eqref{eq: main1} it follows that we can choose $M$ so that $\tau_{2,y}-M t_{\mathrm{LS}} < \rho_y $, and so for some $A \in \con_{1/2}$
\begin{equation}
\label{eq: ynotescaped}
\Pr_y[T_{A^c}>s]>  \pi(A)+\sqrt{\pi(A)\pi(A^c)}.
\end{equation}

\medskip

Denote the  connected component of $x$ in $\cT \setminus \{y\} $ by $A'$.  Since $y \prec x$ we have $\pi(A') \le 1/2$. Hence, for all $z \in A' $ we have  $\Pr_{z }[T_y>\tau_1 ] \le H_{\tau_1}(z,A') \le \pi(A')+1/4 \le 3/4$. Using the Markov property, by induction we get that $\Pr_{z }[T_y>k\tau_1 ] \le (3/4)^k$ for all $k \in \N$ and $z \in A'$ and so $\mathbb{E}_x[T_y] \le 4 \tau_1$.

Let $(v_0=x,v_1,\ldots,v_{k}=y)$ be the path from $x$ to $y$.
Define $ \xi_i:=T_{v_{i}}-T_{v_{i-1}}$. Then by the tree structure, under
$\Pr_x$, we have that $T_{y}=\sum_{i=1}^{k}\xi_i$ and that $\xi_1,\ldots,\xi_k$
are independent. Denote $\Phi(\cT_{v_i}):=\frac{\pi(v_{i})P(v_{i},v_{i+1})}{\pi(\cT_{v_i})}$. By specializing Kac's formula to trees (see \cite[(2.23)]{aldous} for the general Kac's formula we are using and for its specialization for trees see \eqref{eq: Kac} below and c.f.~\cite[Proposition 5.6 and Lemma 5.2]{cutoff}) we have that $\mathbb{E}_{v_{i-1}}[T_{v_{i}}]=1/\Phi(\cT_{v_i})$ and that $\mathbb{E}_{v_{i-1}}[T_{v_{i}}^2] \le 2 \mathbb{E}_{v_{i-1}}[T_{v_{i}}] \mathbb{E}_{\pi_{\cT_{v_{i-1}}}}[T_{v_i}] \le 4 \rel  \mathbb{E}_{v_{i-1}}[T_{v_{i}}]   $. Whence,
\[ \mathrm{Var}_{x}[T_{y}] = \sum_{i=1}^{k} \mathrm{Var}_{v_{i-1}}[T_{v_{i}}] \le  \sum_{i=1}^{k}
\mathbb{E}_{v_{i-1}}[T_{v_{i}}^2] \le 4 t_{\mathrm{rel}} \sum_{i=1}^{k}
\mathbb{E}_{v_{i-1}}[T_{v_{i}}]=4 \rel \mathbb{E}_x[T_y] \le 16 \rel \tau_1 .\]
By Chebyshev inequality
\begin{equation}
\label{Chebyshev1}
\Pr_x[|T_y-\mathbb{E}_x[T_y] | > 32  \sqrt{ \rel \tau_1} ] \le 1/64.
\end{equation}
Let $s':=\max ( \mathbb{E}_x[T_y]-32\sqrt{ \rel \tau_1},0) $. By \eqref{eq: ynotescaped},  \eqref{Chebyshev1}, $s>64 \sqrt{ \rel \tau_1} $ and the Markov property
\[\Pr_x[X_{s+s'} \in A ] \ge \Pr_x[|T_y-\mathbb{E}_x[T_y] | \le 32 \sqrt{ \rel \tau_1}] \times \Pr_y[T_{A^c}>s] > (\pi(A)+\frac{1}{2} \sqrt{\pi(A)\pi(A^c)}) \frac{63}{64},  \]
and so $\Pr_x[X_{s+s'} \in A ] \ge \pi(A)+\frac{1}{8} \sqrt{\pi(A)\pi(A^c)}  $. The proof is concluded using \eqref{eq: Lagrange} as follows. In the notation from \eqref{eq: Lagrange}, $\Pr_x^{s+s'} \in \PP_{A,\delta}$ for some $\delta \ge \frac{1}{8}\sqrt{\pi(A)/\pi(A^c)} $ and thus (using the Poincar\'e inequality) $\| \Pr_x^{s+s'-2 \rel }-\pi\|_{2,\pi} > 4\| \Pr_x^{s+s'}-\pi\|_{2,\pi}  \ge4 \delta \sqrt{\pi(A^{c})/\pi(A)}  \ge 1/2  $.  \qed

\medskip

{\em Proof of Proposition \ref{prop: Tdelta}:} Fix some leaf $x $, $0<\delta  \le 1/4 $ and $A \in \con_{\delta}$. Recall that  $\bar A=A^{c} \setminus D_{\delta}$. Using the tree structure it is easy to see that for all $s,s' \ge 0$ \[\Pr_x[T_{A^{c}}>s+s'  ] \le \Pr_x[T_{\bar A}>s+s' ] \le \Pr_x[T_{x_{\delta}}>s]+   \Pr_{x_{\delta}}[T_{\bar A}>s'] \le \Pr_x[T_{x_{\delta}}>s]+\Pr_{\pi_{\cT_{x_{\delta}}}}[T_{\bar A}>s  '] \] and so by \eqref{eq: stationary<quasi}, the def.~of $b_x$ and the fact that $\pi_{V \setminus \bar A}(\cT_{x_{\delta}}) > 1/2$ (as $\pi(V \setminus \bar A )<2 \delta <2 \pi(\cT_{x_{\delta}}) $)
\begin{equation*}
\begin{split}
&  \Pr_x[T_{A^{c}}>b_x+3\kappa +10 \rel  ] \le \Pr_x[T_{x_{\delta}}>b_x]+\Pr_{\pi_{\cT_{x_{\delta}}}}[T_{\bar A}>3\kappa+10 \rel  ]  \\ & < \Pr_x[T_{x_{\delta}}>b_x]+ 2\Pr_{\pi_{V \setminus \bar A}}[T_{\bar A}>3\kappa+10 \rel  ] \le \delta^3/4 +\delta^3/4 = \delta^3/2  . \quad \qed
\end{split}
\end{equation*}

\medskip

{\em Proof of Proposition \ref{prop: LD}:} By \cite[Corollary 5.5]{cutoff}
we have that $\mathbb{E}_x[T_{x_{\delta}}] \le \tau_1 + C_5 \sqrt{\tau_1 \rel}$ and hence it suffices to show that
\begin{equation}
\label{eq: LD2}
\forall t \in [0,2\mathbb{E}_x[T_{x_{\delta}}]], \quad \Pr_x[T_{x_{\delta}} \ge \mathbb{E}_x[T_{x_{\delta}}] +t ] \le \exp[-t^{2}\la(D_{\delta})/(8 \mathbb{E}_x[T_{x_{\delta}}] ) ].
\end{equation}
\begin{equation}
\label{eq: LD3}
\forall t \ge 2\mathbb{E}_x[T_{x_{\delta}}], \quad \Pr_x[T_{x_{\delta}} \ge \mathbb{E}_x[T_{x_{\delta}}] +t ] \le \exp[- \la(D_{\delta}) t /4].
\end{equation}
Indeed, if $t_{1}:=8\sqrt{\mathbb{E}_x[T_{x_{\delta}}]/\alpha_{_{x}}(\delta)} \le 2\mathbb{E}_x[T_{x_{\delta}}]  $ then by \eqref{eq: LD2} $\Pr_x[T_{x_{\delta}} \ge \mathbb{E}_x[T_{x_{\delta}}] +t_1 ] \le \delta^3/4$. Otherwise, $t_2:=32/\alpha_{_{x}}(\delta)>2\mathbb{E}_x[T_{x_{\delta}}] $, and by \eqref{eq: LD3}, $\Pr_x[T_{x_{\delta}} \ge \mathbb{E}_x[T_{x_{\delta}}] +t_2 ] \le \delta^3/4$.

\medskip

We note that \eqref{eq: LD2} is essentially Lemma 5.8 in \cite{cutoff}. We start with an auxiliary calculation
\begin{claim}
Fix some leaf $x$ and $\delta \in (0,1/4]$. Let $D_{\delta}$ be the connected component of $x$ in $\cT \setminus \{x_{\delta}\}$. Let $y \in D_{\delta}$ and $z$ be its parent. Then for all $\beta \le \la(D_{\delta})/2  $ we have that
\begin{equation}
\label{eq: Laplace}
\mathbb{E}_y[e^{\beta T_z}] \le 1+\mathbb{E}_y[T_{z}] \beta(1+ 2\beta / \la(D_{\delta})) \le e^{\mathbb{E}_y[T_{z}] \beta(1+ 2\beta / \la(D_{\delta})) }.
\end{equation}
\end{claim}
{\em Proof of \eqref{eq: Laplace}:} Let $\Phi(\cT_y):=\frac{\pi(y)P(y,z)}{\pi(\cT_y)}$. Let $f $ and $g$ be the density functions of $T_z$ started from $y$ and $\pi_{\cT_y} $, resp.. By Kac formula (c.f.~\cite[Proposition 5.6]{cutoff} or \cite[(2.23)]{aldous}),
\begin{equation}
\label{eq: Kac}
\forall t \ge 0, \quad  g(t) =\Phi(\cT_y)\Pr_y[T_z>t], \quad \text{and hence} \quad \Phi(\cT_y)\mathbb{E}_y[T_{z}]=1.
\end{equation}
Recall that by \eqref{eq: stationary<quasi} the law of $T_z $ starting from $\pi_{\cT_y} $ is stochastically dominated by the Exponential distribution with parameter $\la(\cT_y) \ge \la(D_{\delta}) $ and so for every non-decreasing function $k$ we have that $\int_{0}^{\infty} k(t)g(t)dt \le \int_0^{\infty} k(t)\la(D_{\delta})e^{-\la(D_{\delta})t}dt $. Finally by \eqref{eq: Kac}
\[\mathbb{E}_y[e^{\beta T_z}] - 1= \int (e^{\beta t}-1)f(t)dt= \int \beta e^{\beta t}\Pr_y[T_z>t]dt=\mathbb{E}_y[T_{z}] \int \beta e^{\beta t}g(t)dt  \]
 \[=\beta\mathbb{E}_y[T_{z}] \int  e^{\beta t}\la(D_{\delta})e^{-\la(D_{\delta})t}dt=\frac{\beta\mathbb{E}_y[T_{z}]\la(D_{\delta})}{\la(D_{\delta})-\beta} \le \mathbb{E}_y[T_{z}] \beta(1+ 2\beta /\la(D_{\delta})), \] where we used  $\beta \le \la(D_{\delta})/2 $ to deduce that $\frac{\la(D_{\delta})}{\la(D_{\delta})-\beta}=1+\frac{\beta}{\la(D_{\delta})-\beta} \le 1+ \frac{2\beta}{\la(D_{\delta})}   $.  \qed

\medskip

We now return to conclude the proofs of \eqref{eq: LD2}-\eqref{eq: LD3}. Let $t \in [0,2\mathbb{E}_x[T_{x_{\delta}}]]$. Set $\beta=\frac{t\la(D_{\delta})}{4\mathbb{E}_x[T_{x_{\delta}}]} $ (note that $\beta \le \la(D_{\delta})/2)  $. Let the path from $x$ to $x_{\delta}$ be $(y_1=x,\ldots,y_r=x_{\delta})$. Observe that starting from $x $ we have that $T_{x_{\delta}}=\sum_{i=2}^{r} T_{y_{i}}-T_{y_{i-1}} $. By the Markov property the terms in the sum are independent and $T_{y_{i}}-T_{y_{i-1}}$ is distributed as $T_{y_{i}}$ started from $y_{i-1}$. Denote $\mu_i:=\mathbb{E}_{y_{i-1}}[T_{y_{i}}] $ and $\mu:=\sum_{i=2}^r \mu_i=\mathbb{E}_x[T_{x_{\delta}}]$. By \eqref{eq: Laplace},  independence and our choice of $\beta$
\[\Pr_x[T_{x_{\delta}} \ge \mu+t ]\le e^{- \beta (\mu+t)} \prod_{i=2}^r \mathbb{E}_{y_{i-1}}[e^{\beta T_{y_i}}] \le e^{- \beta (\mu+t)} \prod_{i=2}^r e^{\mu_i \beta(1+ 2\beta / \la(D_{\delta})) }= e^{-t^{2}\la(D_{\delta})/(8\mu)}.   \] The proof of \eqref{eq: LD3} is analogous, now with the choice $\beta=\la(D_{\delta})/2 $. \qed
\section{Open Problems}

The \emph{modified Log-Sobolev constant} is defined as
\[ c_{\mathrm{MLS}}:=\inf_{f \in \R^{\Omega} }\mathcal{E}(e^{f}, f)/\mathrm{Ent}_{\pi}(e^{f}). \]

The following question suggests a natural extension of Theorem \ref{thm: kappa=tell}.
Recall that under reversibility $1/c_{\mathrm{LS}} \le 2\tau_{\infty} $ and $\la^{-1} \log 2 \le \tau_1$ (e.g.~\cite[Lemma 20.11]{levin}). The following question asks whether a similar relation holds between $c_{\mathrm{MLS}} $ and $\tau_{\entt}$.

\begin{question}
Is it the case that $1 / c_{\mathrm{MLS}} \le C \tau_{\entt}    $ for some absolute constant $C$?
\end{question}

\begin{question}
Is it the case that under reversibility $1 / c_{\mathrm{MLS}} \le C \rho_{\entt}    $ for some absolute constant $C$ (and thus $1 / c_{\mathrm{MLS}} \le C \tau_{\entt}$)?
\end{question}
\begin{question}
Recall that under reversibility  $ \tau_{2}  \le \rho+C/c_{\mathrm{LS}} $.  Is it true that under reversibility $ \tau_{\entt}  \le \rho_{\entt}+C/c_{\mathrm{MLS}} $? \end{question}

\section*{Acknowledgements}
We are grateful to Shirshendu Ganguly, Gady Kozma, James Lee and Prasad Tetali  for useful discussions. Most of the work on this paper was done while the first author was an intern at Microsoft Research, Redmond. The first author would like to thank Microsoft Research for two wonderful summers.

\nocite{*}
\bibliographystyle{plain}
\bibliography{cutoff}

\vspace{2mm}

\end{document}